\documentclass[10pt]{amsart}

\usepackage{epigraph}
\usepackage[utf8]{inputenc}
\usepackage[T2A]{fontenc}
\usepackage[russian, main=english]{babel}

\usepackage{amsmath,amsthm,amsfonts, bbm, amssymb, hyperref,graphicx,color}
\usepackage{lineno, caption, subcaption}
\newcommand{\ii}{\mathbbm{i}}
\newcommand{\tee}{\mathfrak{t}}
\newcommand{\ai}{\mathfrak{i}}

\renewcommand{\C}{\mathbb{C}}
\newcommand{\CW}{\mathcal{C}_{\mathbb{W}}}

\newcommand{\Hyp}{\mathbb{H}}

\newcommand{\R}{\mathbb{R}}
\newcommand{\Z}{\mathbb{Z}}

\newcommand{\W}{\mathbb{W}}
\newcommand{\height}{\operatorname{ht}}
\newcommand{\htk}{\hat{T}|_{\overline{K}}}
\newcommand{\id}{\operatorname{id}}
\newcommand{\X}{\mathbb X}

\newcommand{\norm}[1]{\left\vert #1 \right \vert}	
\newcommand{\Norm}[1]{\left\Vert #1 \right \Vert}
\newcommand{\floor}[1]{\left\lfloor #1 \right\rfloor}

\renewcommand{\Re}{\text{Re}}
\renewcommand{\Im}{\text{Im}}

\newcommand{\Isom}{\text{Isom}}
\newcommand{\Stab}{\text{Stab}}

\newcommand{\rad}{\operatorname{rad}}

\def\[#1\]{\begin{linenomath}\begin{align}#1\end{align}\end{linenomath}}
\def\(#1\){\begin{linenomath}\begin{align*}#1\end{align*}\end{linenomath}}

\newcommand{\ignore}[1]{}

\newtheorem{thm}{Theorem}

\newtheorem{prop}[thm]{Proposition}
\newtheorem{lemma}[thm]{Lemma}
\newtheorem{cor}[thm]{Corollary}
\newtheorem{question}{Question}

\numberwithin{thm}{section}

\theoremstyle{definition}
\newtheorem{defi}[thm]{Definition}

\theoremstyle{definition}
\newtheorem{example}[thm]{Example}
\newtheorem{remark}[thm]{Remark}

\numberwithin{equation}{section}

\newcommand{\st}{\;:\;}
\newcommand{\mst}{\;:\;}

\newcommand{\Mod}{\mathcal M}

\newcommand{\Zee}{\mathcal Z}
\newcommand{\ZeePrime}{\Zee'}

\newcommand{\Sph}{\mathbb S}

\title[Ergodicity of Iwasawa Continued Fractions]{Ergodicity of Iwasawa continued fractions via markable hyperbolic geodesics}
\author[A. Lukyanenko]{Anton Lukyanenko}
\address{
Department of Mathematics\\
George Mason University\\
4400 University Drive, MS: 3F2\\
Fairfax, Virginia 22030}
\email{anton@lukyanenko.net}
\author[J. Vandehey]{Joseph Vandehey}
\address{
Department of Mathematics\\
University of Texas at Tyler\\
Tyler, TX 75799
}
\email{jvandehey@uttyler.edu}


\begin{document}

\date{\today}

\begin{abstract}
We prove the convergence and ergodicity of a wide class of real and higher-dimensional continued fraction algorithms, including folded and $\alpha$-type variants of complex, quaternionic, octonionic, and Heisenberg continued fractions, which we combine under the framework of Iwasawa continued fractions. The proof is based on the interplay of continued fractions and hyperbolic geometry, the ergodicity of geodesic flow in associated modular manifolds, and a variation on the notion of geodesic coding that we refer to as geodesic marking. As a corollary of our study of markable geodesics, we obtain a generalization of Serret's tail-equivalence theorem for almost all points. The results are new even in the case of some real and complex continued fractions.
\end{abstract}

\maketitle

\epigraph{\textit{``\ldots attempts to find a precise relation between the cutting sequence
of [a geodesic] $\gamma$ and the continued-fraction expansions of endpoints of suitable lifts of $\gamma$ are
fraught with minor discrepancies.''}}{---Caroline Series \cite{Series1}}
\section{Introduction}
\subsection{Background}

Since the early work by Lagrange and Gauss linking regular continued fractions (CFs) to algebra and dynamical systems, an extensive and ongoing effort has focused on expanding the scope of CF theory to new algorithms. While regular CFs represent the fractional part $x-\floor{x}$ of a real number $x\in \R$ as a descending iterated fraction \[\cfrac{1}{a_1+\cfrac{1}{a_2+\cdots}}\] with positive integer digits, a menagerie of one-dimensional CF variants have been formed by modifying various aspects of this simple construction: whether by changing positive quantities to negative, altering the set of allowable digits, or selecting a different set of numbers to have expansions. (See \S \ref{subsec:landscape} for an introduction to many of these variants.) 

After over 200 years of study, the one-dimensional CFs are largely well-understood. Most of them inherit the essential properties of regular CFs from the viewpoints of algebra, dynamics, and geometry: Lagrange's Theorem, shift map ergodicity, and Diophantine interpretation, respectively. The study of one-dimensional CFs has been facilitated by a connection to hyperbolic geometry, pioneered by Artin \cite{Artin} and developed by Series \cite{Series1}, Katok and Ugarcovici \cite{KU2005,KU2007,KU2012}, and others. In particular, Artin observed that the Gauss map for regular CFs can be identified with a section of geodesic flow in a finite cover of the modular surface; leading to extensive developments in both CF theory and the study of geodesics on hyperbolic manifolds.

In trying to extend these properties beyond one-dimensional CFs, one is immediately confronted by the question of how to generalize one-dimensional CFs to more than one dimension. Several algorithms, such as those of Jacobi-Perron, Brun, and Selmer, act by building up the CF expansion to several different real values simultaneously \cite{SchweigerBookmulti}. Other algorithms, such as the Hurwitz complex CF algorithm \cite{MR1554754} or the Heisenberg CF algorithm studied previously by the authors \cite{LV}, treat points in these spaces as single entities with a single continued fraction expansion.\footnote{Yet another type of CF-like algorithm deriving more from geometric properties can be seen in \cite{Hayward,HP}.} This is analogous to how complex points can be understood either via their real and complex part (i.e., essentially in $\mathbb{R}^2$) or as an element in complex space (in $\mathbb{C}$). In this paper we will generally be interested in the latter form of higher-dimensional CF expansion, as it has a more natural connection to hyperbolic geometry.

The story of these higher-dimensional CFs has been markedly different from their real CF cousins. Despite interest in these topics stretching back to the 1850's \cite{Hamilton1,Hamilton2}, only a small number of algorithms are known to be well-behaved. Among them is the A. Hurwitz complex CF \cite{MR1554754}, which represents a complex number $z$ with real and imaginary part both in $[-1/2,1/2)$ as a descending iterated fraction 
\[
\cfrac{1}{a_1+\cfrac{1}{a_2+\cfrac{1}{a_3+\dots}}}, \qquad a_i \in \mathbb{Z}[\ii]\setminus \{0,\pm 1,\pm\ii\}.
\]
(See \ref{subsubsection:ComplexCFs} for a full description.) Proofs of, for example, ergodicity for these well-behaved algorithms are extremely delicate \cite{NakadaErgodic}: the space of the algorithm has a serendipitous decomposition, which results in a finite range property among other features, and this allows high-powered results (such as those in \cite{IY,SW}) to be applied. Should the algorithm be perturbed, even slightly (see \S\ref{subsubsection:ComplexCFs} and Figure \ref{fig:hurwitz}), the decomposition will break down and the methods will no longer apply.

As a major goal of this paper is to prove properties like ergodicity for a larger variety of higher-dimensional CFs (including perturbed variations of standard algorithms), let us discuss some of the roadblocks to using traditional techniques. First of all, one does not expect the structure of the cylinder sets (the sets of numbers whose expansion all start with the same sequence of digits) to have a simple structure, so methods like those cited above will not apply. Second, the natural extension of even some one-dimensional CF variants  (see \cite[\S 7]{AS}) as well as simple higher-dimensional CFs (see \cite{EINN,HV}) is already fractal in nature, which makes it difficult to prove results about the natural extension, let alone about the simpler algorithm. Third, making a precise connection between CF digits and geodesic coding is ``fraught with minor discrepancies'' and in its strong form would imply properties such as Serret's tail equivalence theorem \cite{Panti} that are known to fail for higher-dimensional CFs. Indeed, the geodesic coding approach has long been considered ``intrinsically two-dimensional'' \cite{AF84}.

In this paper, we develop a softer version of geodesic coding, which we refer to as \emph{geodesic marking}. In a typical\footnote{What we describe here is an arithmetic coding, in the terminology of Katok and Ugarcovici \cite{KU2005}. Codings formed by cutting sequences are related.} geodesic coding, we look at a given geodesic from two perspectives: first, we have a bi-infinite sequence formed by the continued fraction digits of both the forward and backward endpoints of the geodesic, and second, we have a bi-infinite sequence of intersections of our geodesic with a particular cross-section. A shift in one sequence should correspond to a shift in the other. In particular, returning to the cross-section after flowing along the geodesic should move the CF expansion forward one digit. Our geodesic \emph{marking} still has the two bi-infinite sequences, but now returning to the cross-section can move the CF expansion forward several digits at a time. Thus the first-return map to the cross-section now corresponds to a jump transformation for the continued fraction.  This jump transformation, in practice, skips over strings of small digits.\footnote{What counts as a small digit could be made effective, but we do not do so here.} It should be noted that small digits appear to cause some of the roadblocks mentioned above: cylinder sets associated to small digits tend to be irregular, while those of large digits are far better behaved, for instance. So, in essence, geodesic markings skip over the troublesome parts of CF algorithms.

Geodesic marking provides a robust connection between these higher-dimensional CFs and hyperbolic geometry which is preserved even under perturbation of the algorithm. The following theorem illustrates some significant cases our work applies to:
\begin{thm}
\label{thm:super-main}
Folded complex CFs, folded Hurwitz quaternionic CFs, folded octonionic CFs, and folded Heisenberg CFs as well as their $\alpha$-type variants are convergent and ergodic.
\end{thm}
In particular, this illustrates how our work applies to several different spaces (complex numbers, quaternionic numbers, octonionic numbers, and the Heisenberg group) and many systems within those spaces (folded and $\alpha$-type variants are discussed in more detail in \S \ref{subsec:landscape}, see also Figure \ref{fig:hurwitz}). Our results also apply to several one-dimensional CF algorithms, such as folded real CFs and some of Nakada's $\alpha$-CFs (see \S\ref{subsubsec:nearest} and \S \ref{subsubsec:folded}).

While convergence follows standard arguments, the ergodicity statement is a substantial breakthrough for higher-dimensional CFs, where it was only previously known for specific complex CF variants, such as the A.~Hurwitz and J.~Hurwitz CFs. 
Our approach furthermore provides a flexible, unifying method for understanding both one-dimensional and higher-dimensional CFs.

Theorem \ref{thm:super-main} follows from a more general result concerning CFs on boundaries of rank-one symmetric spaces of non-compact type, which we refer to as \emph{Iwasawa inversion spaces}. CFs were first extended to this setting by the authors in \cite{LV}, where a CF theory on the non-commutative Heisenberg group was proposed. In \cite{MR4126256}, Chousionis-Tyson-Urbanski, studying conformal iterated function systems, defined Iwasawa continued fractions on the closely-related \emph{Iwasawa groups} (see \S \ref{sec:further}). Here, we extend the definition of Iwasawa CFs to an Iwasawa CF \emph{algorithm} associating a digit sequence to each point in an Iwasawa inversion space and leverage the connection to hyperbolic geometry to prove the following theorem:

\begin{thm}\label{thm:main}
Every discrete and proper Iwasawa CF is convergent. Moreover, if it is complete, then it is ergodic.
\end{thm}

We will postpone the full definitions of these terms until \S \ref{sec:defn}, but will provide some insight into them now. Discreteness simply says that the modular group $\Mod$ associated to our CF algorithm acts discretely on the corresponding hyperbolic space. It is necessary to ensure that we have a finite-volume hyperbolic manifold (generalizing the modular surface) in which to look at geodesic flow. Properness says that the only points under consideration for our CF algorithm have norm bounded away from 1. This guarantees that CF expansions converge quickly, among other properties. Properness also helps us avoid indifferent fixed points in our dynamical system, which have been noted before to cause infinite invariant measures \cite{DHKM}. Completeness says that the set of digits for our CF algorithm is maximal in an appropriate sense, the upshot of which is that the sequence of digits in the CF expansion of a point is functionally the only expansion the hyperbolic geometry can see.

In the case where a system is not complete, we can still obtain a partial result:

\begin{thm}\label{thm:notmain}
Let $T: K\rightarrow K$ be the shift map for a discrete and proper Iwasawa CF with $n\geq 1$ central symmetries. Then $T$ has at most $n$ ergodic components.
\end{thm}

A full description of central symmetries will appear in \S\ref{defi:centrallySymmetric}. For the moment, we can consider centrally symmetric systems as ones where the system is incomplete due to the appearance of hidden symmetries, such as $x\mapsto -x$, as in \S\ref{subsubsec:folded}.

The use of geodesic marking, as opposed to classical geodesic coding, is critical to Theorem \ref{thm:main}. As noted above, one typical corollary of geodesic coding is Serret's tail-equivalence theorem for every point. That is, two points lie in the same orbit of the modular group $\Mod$ if and only if the tails of their CF digit sequences agree. However, for the A.~Hurwitz complex CFs, Lakein \cite{Lakein} provides an explicit counterexample to tail-equivalence.\footnote{Lakein's counterexample makes use of an element not belonging to $\Mod$, but this can be remedied by multiplying his choice of $A$ by $\ii$.} Thus, one would not expect for geodesic coding to be available in this case. Geodesic marking, on the other hand, avoids certain points which exhibit pathological behavior, those with all small CF digits. This allows for an ergodicity result and leads to the following \emph{a.e.}~tail equivalence result for Iwasawa CFs (proven as Theorem \ref{thm:tail}), which is novel for all higher-dimensional algorithms including \emph{folded} A. Hurwitz complex CFs: 

\begin{thm}
\label{thm:IntroTail}
Almost surely, two points in a complete, discrete, and proper Iwasawa CF are tail-equivalent if and only if they are $\Mod$-translates of one another.
\end{thm}

The question of tail-equivalence is being actively researched even for one-dimensional CFs, see \cite{BM,Panti}. The importance of small digits versus large digits to tail-equivalence has been noted before in \cite{NN}.

\subsection{Key examples of CF algorithms}
\label{subsec:landscape}
We now describe a number of well-known variants of continued fractions, primarily in the one-dimensional case, that are of interest to us.
We will discuss the algorithms in an increasing order of complexity (see Figure \ref{fig:algorithms} for a diagram), pointing out the variations that motivate the definition of Iwasawa CFs: namely, the choice of underlying space, inversion, digit sequence, and fundamental domain for the corresponding lattice; as well as the definitions of properness, completeness, and discreteness. A more thorough discussion of the class of Iwasawa CFs is provided in \S \ref{sec:IwasawaCF}, along with a more complete list of known Iwasawa CF algorithms in Table \ref{tab:examples}. 

\begin{figure}
    \centering
    \includegraphics[width=.75\textwidth]{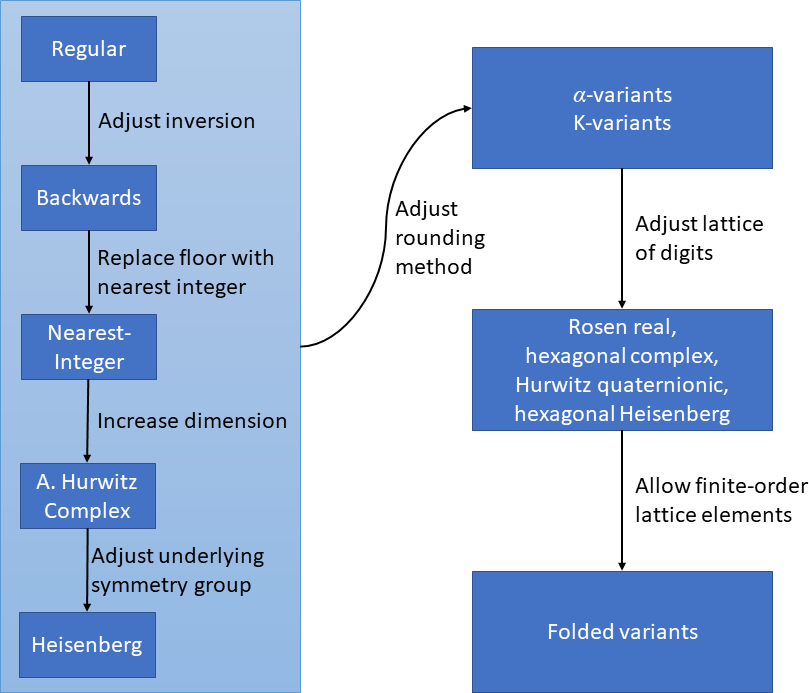}
    \caption{Different Iwasawa CF algorithms can be thought of as variations on the regular CF algorithm: the equations for the shift map, digit extraction algorithm, and recombination algorithm remain the same, while the underlying data is adjusted.}
    \label{fig:algorithms}
\end{figure}

\subsubsection{Regular CFs}
\label{subsubsec:regular}
The \emph{regular continued fraction} representation of a number $x\in [0,1)$ represents it as a limit
$$x=\cfrac{1}{a_1+\cfrac{1}{a_2+\ldots}}$$
where $a_i\in \mathbb{N}$.\footnote{In the introduction, we ignore the behavior of points with finite CF expansion for simplicity.} The digits $a_i$ are extracted from $x$ by repeated applications of the Gauss map $T(x)=1/x-\lfloor 1/x\rfloor$:
$$a_i = \left\lfloor \frac{1}{T^{i-1}x} \right\rfloor.$$
The Gauss map is famously ergodic with an invariant measure given by the density $\frac{1}{\log 2}\frac{1}{1+x}$. (See \cite{DKbook} for a fuller treatment.)

In the framework of Iwasawa CFs, regular CFs are described using the following data:
\begin{enumerate}
    \item the underlying space $X$ is $\R$,
    \item the inversion used is $\iota(x)=1/x$,
    \item the allowed digits are elements of the lattice $\Zee=\Z$,
    \item the set of ``fractional points'' is $K=[0,1)$, which tiles $\R$ under integer translations. 
\end{enumerate}

As we show below, many standard and novel algorithms can be described by adjusting the above data and leaving the formulas above essentially unchanged.

\subsubsection{Backwards CFs}\label{subsec:backwardsCF}
The \emph{backwards CF} (sometimes called R\'enyi CF) reverses the domain of the Gauss map to produce the R\'enyi map $$T_R(x)=T_G(1-x)=\frac{1}{1-x}-\left\lfloor \frac{1}{1-x}\right\rfloor.$$
The CF digits of $x\in [0,1)$ are then extracted in an analogous manner to the one used for standard CFs, via
\begin{equation}
\label{eq:backwardsdigits}
a_i = \left\lfloor \frac{1}{1-T^{i-1}x} \right\rfloor,
\end{equation}
and recombined as
$$x=1-\cfrac{1}{a_1+1-\cfrac{1}{a_2+\ldots}}.$$

The shift map $T_R$ is ergodic, but due to the presence of indifferent fixed points, the corresponding invariant measure is infinite  \cite{AF84}.

With a small adjustment, backwards CFs fit into the framework of Iwasawa CFs, as follows.

The mapping $x\mapsto (1-x)$ conjugates backwards CFs to an equivalent system known as  the $D$-backwards CF with $D=[0,1)$, see  Masarotto \cite{Masarotto}. The resulting shift  map is then given by $T_D(x)=\frac{-1}{x}-\lfloor \frac{-1}{x}\rfloor$.
Adjusting Masorotto's notation by using negative integer digits $a_i<-1$ we take
\begin{equation}
\label{eq:backwardsdigitsright}
    a_i = \left\lfloor \frac{-1}{T^{i-1}x} \right\rfloor.
\end{equation}
and recombine the digits as
$$x=\cfrac{-1}{a_1+\cfrac{-1}{a_2+\ldots}}.$$

All three CF algorithms discussed so far are real algorithms looking at points in $[0,1]$ which use integers for their digits. The only difference between them is the choice of inversion: $x\mapsto \frac{1}{x}$ for regular CFs, $x\mapsto \frac{1}{1-x}$ for backwards CFs, and $x\mapsto \frac{-1}{x}$ for $D$-backwards CFs. 

In the Iwasawa CF formalism, we will assume that inversions send $0$ to $\infty$ and preserve the unit circle. While the backwards CF algorithm a priori doesn't fit this requirement, the conjugate $D$-backwards system is an Iwasawa CF.

We will make use of both of the allowed inversions $\iota_+(x)=1/x$ and $\iota_-(x)=-1/x$ throughout the paper. 

Interestingly, backwards continued fractions are the more natural system within the framework of Iwasawa CFs. The inversion $\iota_-$ is an orientation-preserving linear-fractional mapping, and is an element of the modular group $PSL(2,\Z)$, while $\iota_+$ is orientation-reversing, which forces us to consider the larger group $PGL(2,\Z)$. This leads us to the question of \emph{completeness} of the digit set, see \S \ref{subsubsec:folded}.

\subsubsection{Nearest-integer and $\alpha$-type CFs}
\label{subsubsec:nearest}

The next set of CF algorithms adjusts the set of ``fractional'' points and the corresponding rounding method, while also allowing variation in the choice of inversion.

A nearest-integer CF replaces the unit interval $[0,1)$ with the interval $[-1/2, 1/2)$, and the floor function $\lfloor \cdot \rfloor$ with the nearest-integer mapping $[\cdot]$.  There are three standard systems known as nearest-integer CFs, of which two fit directly into the Iwasawa CF framework, and the third is semi-conjugate to an Iwasawa CF. The first two systems are constructed by  choosing the inversion function $\iota$ to be either $\iota_+(x)=1/x$ or $\iota_-(x)=-1/x$. The corresponding shift map is given by $T(x)=\iota(x)-[\iota(x)]$, and one has the digits $a_i=[\iota(T^{i-1}x)]$, which are still integers. The third system is based on the shift map $T(x)=\norm{1/x}-[\norm{1/x}]$ and a more complicated system of digits, which we will discuss more extensively in \ref{subsubsec:folded}. Due to non-injectivity of the mapping $x\mapsto [\norm{1/x}]$, this third variant does not fit the Iwasawa CF framework.

The $\alpha$-type CFs, with $\alpha\in[0,1]$, form a family of CF algorithms that interpolate between regular and nearest-integer CFs by
operating with the interval $[-\alpha,1-\alpha)$ and the corresponding rounding function $[x]_\alpha=[x+\alpha]$. The forward shift is given by $T(x)=\iota(x)-[\iota(x)]_\alpha$, where $\iota$ is chosen from $\iota_+$, $\iota=\iota_-$, or $x\mapsto \norm{1/x}$. As above, the first two choices fit the Iwasawa CF framework, while the third variant does not. 
All three families of systems are known to be ergodic for all $\alpha\in[0,1]$:
\begin{itemize}
    \item Ergodicity of the $\iota_+$ variant for $\alpha\notin\{0,1\}$ follows from our results and for $\alpha\in[0,1]$ was simultaneously shown by \cite{NS},
    \item Ergodicity of the $\iota_-$ variant for $\alpha\notin\{0,1\}$ is new in this paper (cf.~\cite{AS}). The cases $\alpha=0$ and $\alpha=1$ are also ergodic, since  $\alpha=0$ gives the backwards CF and  $\alpha=1$ gives a system that is conjugate to the regular CF.
    \item Ergodicity of the $x\mapsto\norm{1/x}$ variant was recently proven in \cite{NS}.
\end{itemize}

Generalizing further, one can replace the unit interval with any measurable set $K$ that tiles $\R$ under integer translations and write $T(x)=\iota(x)-[\iota(x)]_K$ where $[x]_K$ denotes the unique integer satisfying $x-[x]_K \in K$. Such systems fall under the framework of Iwasawa CFs. Our results imply that the CF is convergent and the shift map is ergodic as long as $K$ is proper: that is, the closure of $K$ is contained in the open unit ball $(-1,1)$. Note that regular and backward CFs are not proper, but are nonetheless convergent and ergodic.

\subsubsection{Folded CFs}
\label{subsubsec:folded}
We now discuss in more detail the nearest-integer system based on the shift map $T(x)=\norm{1/x}-[\norm{1/x}]$. Because the mapping $x\mapsto \norm{1/x}$ is 2-to-1, the standard approach is to keep track both of the integer digit and the choice made when taking the absolute value
\(b_i = [\norm{1/x}], \hspace{2in} c_i=\text{sign}(x),
\)
so that one reconstructs 
\(
x = \cfrac{c_1}{b_1+\cfrac{c_2}{b_2+\cdots}}.
\)

To maintain similarity to previous algorithms, we combine the integer digit $b_i$ and the sign $c_i$ into a single datum, namely the linear mapping $a_i(x)=c_i(x+b_i)$. This allows us to rewrite the fraction in the format
\( x = \cfrac{1}{a_1\left(\cfrac{1}{a_2(\cdots)}\right)} = \lim_{n\rightarrow \infty}\iota_+a_1\iota_+a_2\cdots\iota_+a_n(0),
\)
where each $a_i$ is now a function and we take the convention of suppressing parentheses and composition signs.

We thus transition from thinking of digits as elements of $\Z$ to thinking of them as automorphisms of $\R$. In the Iwasawa CF framework, we will assume that these automorphisms are isometries of the underlying space, which is indeed the case here. 

Therefore, for the algorithm under discussion, we are  now interested in digits in the expanded lattice $\Zee$ generated by integer translations and negation, i.e., $\Zee=\langle x\mapsto x+1, x\mapsto -x\rangle$.

Inconveniently, moving the set of ``fractional'' points $K=[-1/2, 1/2)$ around by the group $\Zee$ causes overlaps, and we therefore exclude this CF variant from the class of Iwasawa CFs. 

Adjusting to the interval under consideration to $K=[0,1/2]$ provides a non-overlapping tiling of $\R$ (that is, $K$ is a fundamental domain for the action of $\Zee$), giving the \emph{folded CF} (see Marmi-Moussa-Yoccoz \cite{MMY}) that now does fit in the Iwasawa CF framework.

The folded CF algorithm is defined by the following data:
\begin{enumerate}
    \item the underlying space $X$ is $\R$,
    \item the inversion used is $\iota_+(x)=1/x$,
    \item the group $\Zee$ of allowed digits is  generated by $x\mapsto x+1$ and $x\mapsto -x$,
    \item the set of ``fractional points'' is $K=[0,1/2)$, which tiles $\R$ under the action of $\Zee$. 
\end{enumerate}

Given this data, we obtain a rounding function $x\mapsto [x]\in \Zee$ that now provides the unique linear mapping  $[x]\in \Zee$ combining an integer translation and possibly a negation such that $[x]^{-1}(x) \in [0,1/2)$. For example, we have that $[5.1](x)=5+x$ and $[5.1]^{-1}(x) = x-5$, while $[5.7](x)=-(x-6)$ and $[5.7]^{-1}(x) = -(x-6)$. 

For a point $x\in [0,1/2)$, we can then write the forward shift map as $T(x)=[1/x]^{-1}(1/x)$ and extract the digits as $a_i = [1/T^{i-1}(x)]$. The point $x$ is  reconstructed from the digits by writing $x = \lim_{n\rightarrow \infty} \cfrac{1}{a_1\left(\cfrac{1}{a_2(\cdots a_n(0))}\right)}$, or more compactly as $x=\lim_{n\rightarrow \infty} a_1 \iota a_2  \iota \cdots a_n(0)$.

The absolute value mapping from $(-1/2,1/2)$ to $(0,1/2)$ then provides a semiconjugacy between the $\norm{\cdot}$-based nearest-integer fractions and the folded CFs. Ergodicity passes down (but not up!) through semiconjugacy, so folded CFs are ergodic. See Marmi-Moussa-Yoccoz \cite{MMY} for the corresponding invariant measure.

While Marmi-Moussa-Yoccoz do describe folded variants of all $\alpha$-CFs, it is only the nearest-integer variant $\alpha=1/2$ that fits within the Iwasawa CF framework, since the other systems continue to operate with fractional sets $K$ that are not fundamental domains for any relevant lattice.

As it turns out, the folded CFs also arise naturally from the regular CF construction, where we have $\Zee=\Z$ and $\iota=\iota_+$. Since the shift map combines both elements of $\Zee$ and the mapping $\iota$, analysis of the shift map revolves around understanding the group $\Mod = \langle \Z, \iota_+\rangle$. The group $\Mod$ includes the negation mapping $x\mapsto -x$ since 
\[\label{eq:invexample}\dfrac{1}{1+\dfrac{1}{-1+\dfrac{1}{1+x}}}=-x,\]
so that the subgroup $\Zee'\subset \Mod$ of linear transformations (that is, the stabilizer of $\infty$) is the group $\ZeePrime=\langle x\mapsto x+1, x\mapsto -x\rangle$. We thus have that the group of allowed digits $\Zee=\Z$ is smaller than the natural group $\ZeePrime$ of linear transformations, giving what we call an \emph{incomplete} system. Expanding the set of digits to $\langle \Z, x\mapsto -x\rangle$ while also contracting the fundamental domain to $[0,1/2)$ provides a completion of the system, again giving us the folded fractions.

\subsubsection{Rosen CFs}
We finish the discussion of 1-dimensional CFs with the Rosen CFs, whose definition is motivated by connections to hyperbolic geometry of triangle groups. 

To define Rosen CFs, one takes the group $\Zee$ of allowed digits to be $(2 \cos \frac{\pi}{q}) \Z$, and the set of ``fractional points'' to be $K=[-\cos \frac{\pi}{q}, \cos \frac{\pi}{q})$.

Together with the inversion $\iota_-$, the lattice $\Zee$ generates a Hecke group, which acts discretely on the hyperbolic plane (with the case $q=2$ reducing to the modular group $PSL(2,\Z)$). 

From here, the choice of $\iota=\iota_-$ (as used by \cite{MS})
would provide an Iwasawa CF algorithm; and our results imply that the corresponding shift map is ergodic. We emphasize that the discreteness of $\Mod=\langle \Zee, \iota_-\rangle$ within the isometry group of hyperbolic space plays a key role in our proof, and that other choices of multiplier in front of $\Zee$ would yield badly-behaved systems.

Lastly, we note that Rosen's original CF algorithm instead is based on the mapping $x\mapsto \norm{1/x}$, and is shown to be ergodic (in fact, weak Bernoulli) in \cite{BKS}. This algorithm is not encompassed by the Iwasawa CF framework.

\subsubsection{Complex CFs}\label{subsubsection:ComplexCFs}

We now briefly touch on higher-dimensional CFs, in the planar case. For more higher-dimensional CFs, including quaternionic and Heisenberg CFs, see the discussion in \S \ref{sec:IwasawaCF}.

Our primary example is the A.~Hurwitz complex CF, first defined in \cite{Hurwitz}. It is described by the following Iwasawa CF data:
\begin{enumerate}
    \item the underlying space $X$ is $\C$,
    \item the inversion used is $\iota(z)=1/z$,
    \item the group $\Zee$ of allowed digits is the group of Gaussian integers, $\Z[\ii]$,
    \item the specified fundamental domain $K$ of $\Zee$ is the unit square centered at the origin.
\end{enumerate}
Thus, the shift map is given by $T(z)=1/z - [1/z]$ where $[\cdot]$ finds the nearest Gaussian integer; the digits are extracted via $a_i = [1/T^{i-1}z]$, and reconstructed as 
\(
z=\cfrac{1}{a_1+\cfrac{1}{a_2+\cdots}}
\)
It is common to write the system in real coordinates, with corresponding data:
\begin{enumerate}
    \item the underlying space $X$ is $\R^2$,
    \item the inversion used is $\iota(x,y)=\frac{(x,-y)}{x^2+y^2},$ (we will denote such conjugate-reflections by $\iota_c$),
    \item the group $\Zee$ of allowed digits is the group $\Z^2$,
    \item the specified fundamental domain $K$ of $\Zee$ is the unit cube centered at the origin, i.e.~$[-1/2,1/2)\times [-1/2,1/2)$.
\end{enumerate}
Both the real and complex descriptions of the Hurwitz CF are quite natural: the mappings $\iota$ and $\Z^2$ both lift to isometries of \emph{real} hyperbolic 3-space, while the corresponding modular group $\Mod=\langle \Zee, \iota\rangle$ is shown to be discrete by embedding into $PSL(2,\Z[\ii])$, see Proposition \ref{prop:firstcomplete}.

Ergodicity of the Hurtwitz CF was shown by Nakada in \cite{NakadaErgodic} (cf, \cite{Hensley}).

As in the case of $\iota_+$ real CFs, the system is not complete, since the stabilizer of $\infty$ in $\Mod$ contains the unexpected mapping $z\mapsto -z$:   
\(\dfrac{1}{1+\dfrac{1}{-1+\dfrac{1}{1+z}}}=-z.\)
As in the case of real folded fractions, we can create a folded variant by extending $\Zee$ to include negation and reducing $K$ correspondingly. For example, one could take $K=[-1/2,1/2)\times [0,1/2)$. In general, we will call a CF algorithm a \emph{folded variant} if $\Zee$ is expanded to the stabilizer of $\infty$ in $\Mod$, and $K$ is similarly reduced.

One can likewise create $\alpha$-type variants by shifting the location of the fundamental domain, i.e. replacing $K$ with $K+\alpha$; or create more exotic variants by choosing a different fundamental domain entirely, e.g., by choosing a tetromino to create the Tetris CFs. However, it is not the case that we can arbitrarily shift folded variants. For example, the set $[-1/2,1/2)\times [-1/4,1/4)$ is not a fundamental domain for the group $\langle \Z[\ii], (x,y)\mapsto (-x,-y)\rangle$.

See Figure \ref{fig:hurwitz} for illustrations of these algorithms and some of their cylinder sets. The finite range condition appears to fail in these cases. We recover ergodicity for folded variants. For centrally-symmetric systems like the Tetris variant, we are able to bound the number of ergodic components by 2. The $\alpha$-variant with $\alpha=0.3$ shown in the figure is not complete and not centrally symmetric with respect to $z\mapsto -z$, and thus none of the results of this paper apply to it.

\begin{figure}[ht]
\begin{subfigure}[b]{0.4\textwidth}\includegraphics[width=\textwidth]{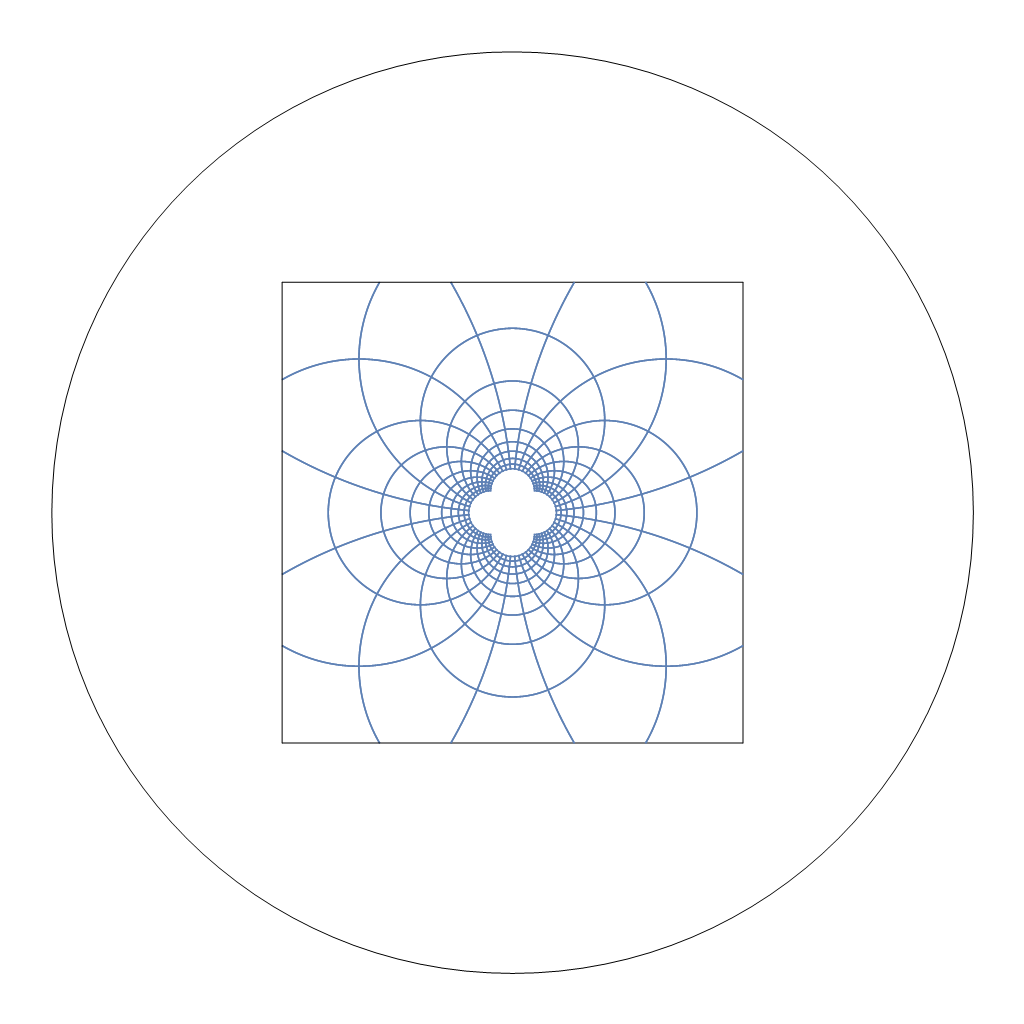}\caption{Hurwitz CF}\end{subfigure}
\begin{subfigure}[b]{0.4\textwidth}\includegraphics[width=\textwidth]{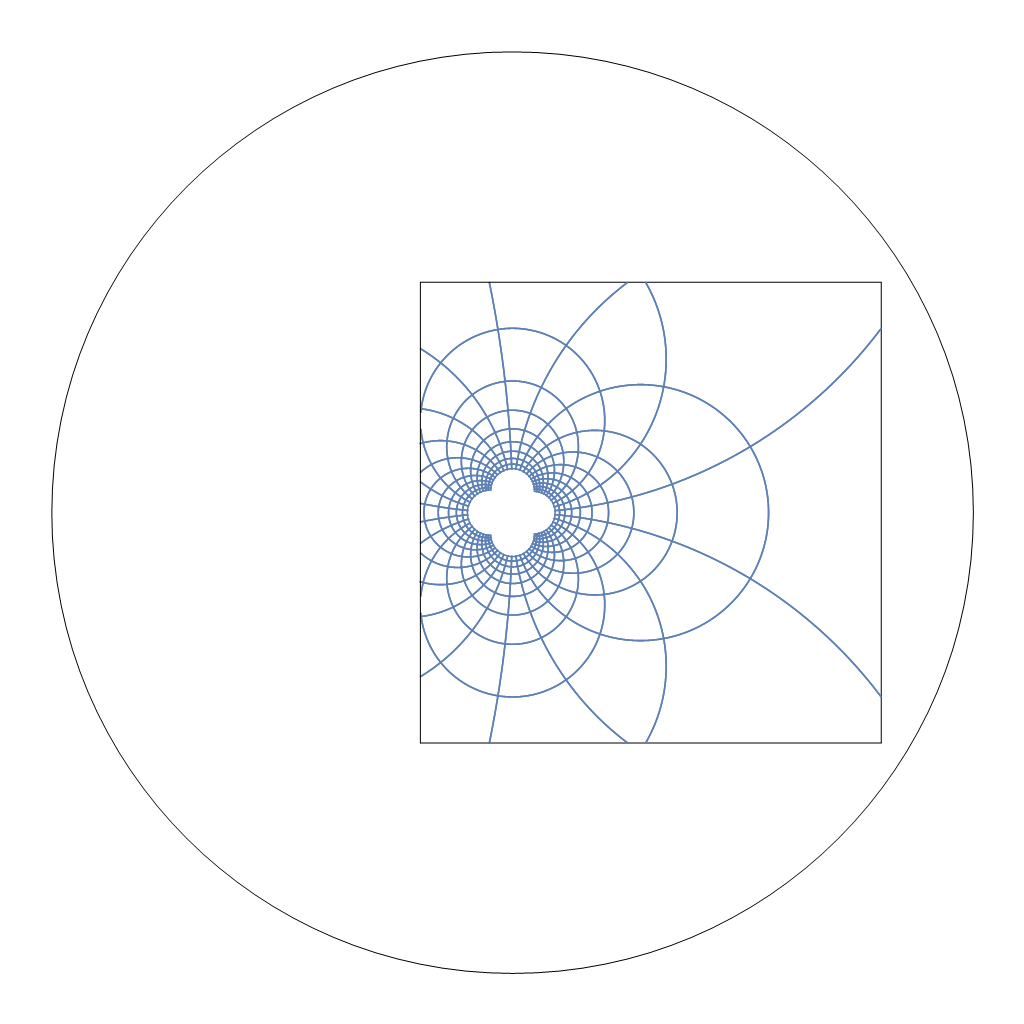}\caption{$\alpha$-variant for $\alpha$=0.3}\end{subfigure}\\
\begin{subfigure}[b]{0.4\textwidth}\includegraphics[width=\textwidth]{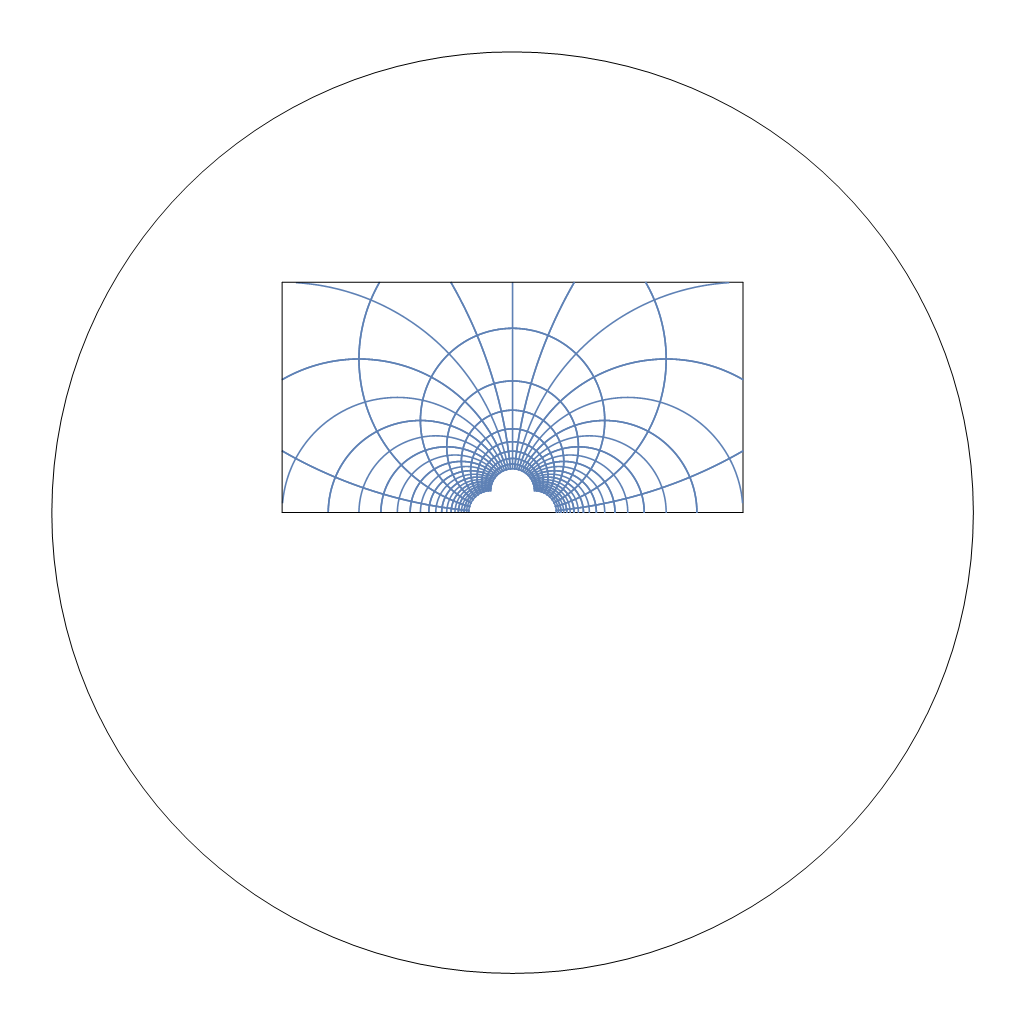}\caption{Folded variant}\end{subfigure}
\begin{subfigure}[b]{0.4\textwidth}\includegraphics[width=\textwidth]{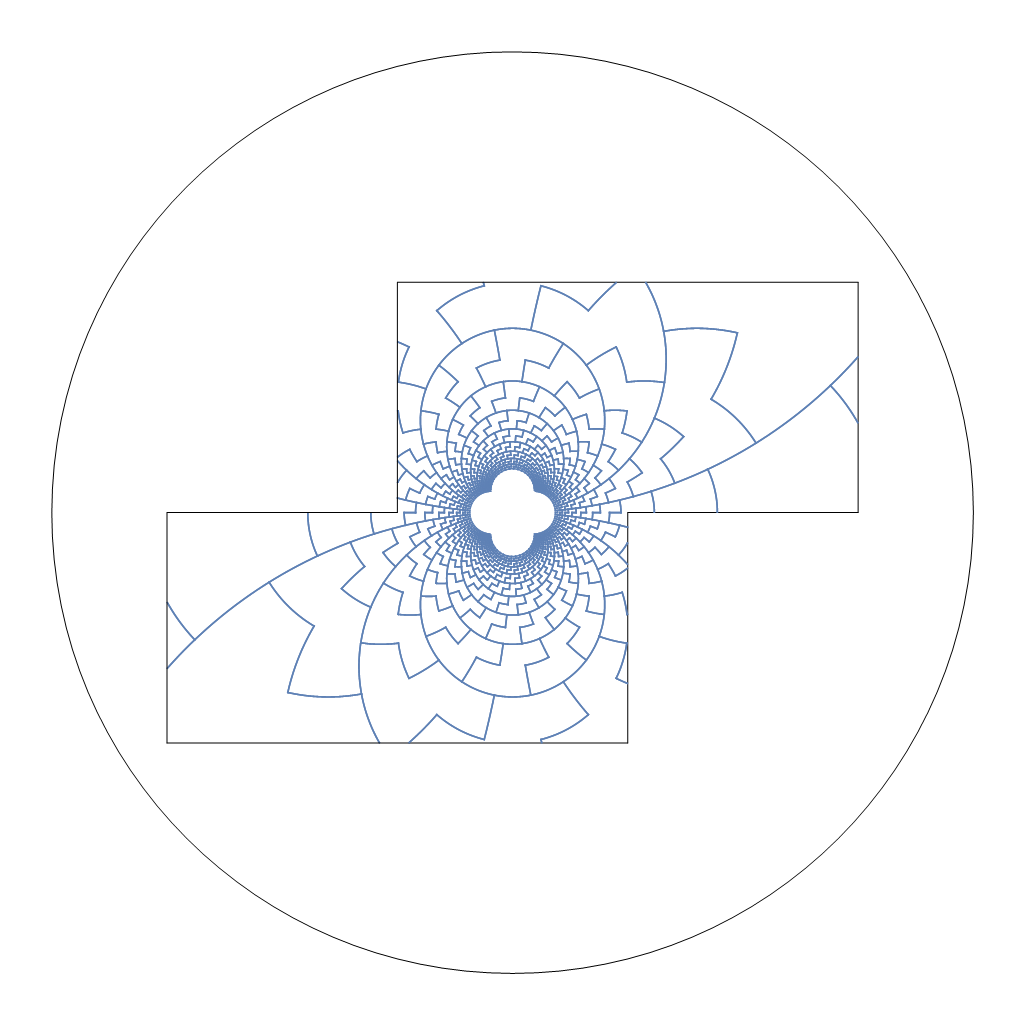}\caption{Tetris variant}\end{subfigure}
\caption{Four variants of the Hurwitz complex CF algorithm. The fundamental domain $K$ in each case is displayed inside the unit circle (fixed by the inversion $\iota_c$), and is decomposed into rank-$1$ cylinder sets. The lattice $\Zee=\Z^2$ is extended by the rotation $(x,y)\mapsto(-x,-y)$ in the folded variant.}
\label{fig:hurwitz}
\end{figure}

\subsection{Theorem \ref{thm:main} in a special case}\label{subsec:condensed proof}

We now outline our proof of ergodicity in the case of nearest-integer CFs with inversion $\iota_-$, where some simplifications are possible (cf.~Remark \ref{rmk:what goes wrong}). Ergodicity is certainly not new in the nearest-integer case, and connections to geodesic flow have been used since at least the work of Adler-Flatto \cite{AF84}.\footnote{From a historical perspective, using ergodicity of geodesic flow to prove the ergodicity of a CF alogrithm is backwards. The ergodicity of regular CFs was shown first, and the ergodicity of geodesic flow in the modular surface was proven using this \cite{Artin,Hedlund}.} For a more thorough treatment of these techniques in the regular CF case, we recommend \cite[\S 9.6]{EW}.

We start by viewing $\mathbb{R}$ as the real axis in $\C$, and interpret the upper half-plane as the hyperbolic plane $\Hyp^2_\R$. Both the integer shifts $\Zee$ and the inversion $\iota$ on $\R$ extend to the half-plane, where they now act by isometries. The modular group $\Mod$ generated by $\Zee$ and $\iota$ acts on $\Hyp^2_\R$ discretely, and gives rise to a tiling of the space by translates of the tile $\mathcal T$ bounded by the vertical lines $x=\pm 1/2$ and the unit circle $\Sph$. Notably, each of the lines $x=\pm 1/2$ are equal to $M\Sph$ for an appropriate $M\in \Mod$. We will study hyperbolic geodesics $\gamma$, which takes the form of either a vertical line or a semi-circle that intersects $\mathbb{R}$ at right angles. 

We will derive ergodicity for the CF shift map from the ergodicity of the geodesic flow on the modular surface $\Mod \backslash \Hyp^2_\R$, which we can think of as the tile $\mathcal T$ with ``opposite sides'' identified. That is, the sides $x=\pm 1/2$ are identified by the translation $z\mapsto z+1$, and the two halves of the circular arc at the bottom are identified via $z\mapsto -1/z$. By Mautner's Theorem \ref{thm:Mautner}, geodesic flow in $\Hyp^2_\R$ is ergodic. In particular, a generic geodesic $\gamma$ is dense in $\Mod \backslash \Hyp^2_\R$, see Figure \ref{fig:geodesic and projection}.

It appears to be intuitively clear that, for a geodesic $\gamma \subset \Hyp$, the continued fraction expansion of the forward endpoint $\gamma_+\in \R$ can be immediately read off from the sequence of tiles that $\gamma$ traverses in $\Hyp^2$, or, equivalently, from the sequence of elements of $\Mod$ that are used to normalize it back to the starting tile. Indeed, it appears that the inversion corresponds to $\gamma$ crossing $\Sph$ and the digits count the number of vertical lines crossed before returning to $\Sph$ after an inversion. Our goal will be to formalize this relationship in sufficient detail to prove the ergodicity of the shift map $T$ from the ergodicity of the geodesic flow on $\Mod\backslash\Hyp^2_\R$, doing so without relying on two-dimensional geometry, which has been central to previous approaches. 

Let $\gamma$ be a vertical geodesic as in Figure \ref{fig:geodesic and projection}. Let $a_1=[-1/\gamma_+]$ be the first nearest-integer CF digit of $\gamma_+$ and $M_1^{-1}(z)=-1/z-a_1$ the corresponding element of $PSL(2,\Z)$ enacting the nearest-integer CF shift $T(\gamma_+)$. Applying $M_1^{-1}$ to all of $\gamma$, we obtain Figure \ref{fig:building section}a. We denote the natural elements of $PSL(2,\Z)$ enacting $T^i$ by $M_i^{-1}$.

\begin{figure}[ht]
\includegraphics[width=.95\textwidth]{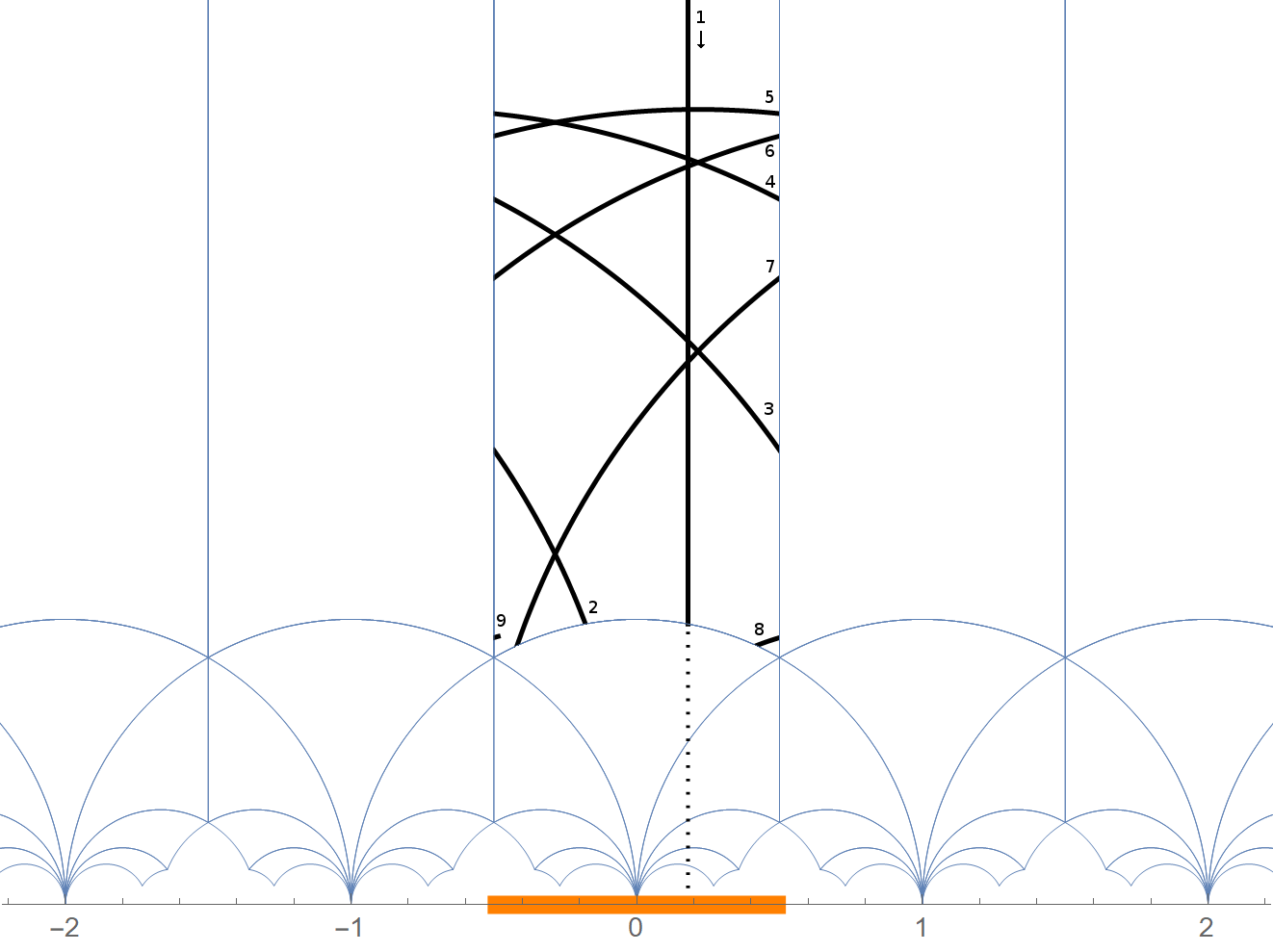}
\caption{The wall-crossings of the vertical geodesic $x=0.1795$ can be used to renormalize it to always stay within the fundamental domain for $\Mod$.}
\label{fig:geodesic and projection}
\end{figure}

Consider now the subsegment $\gamma'$ of $\gamma$ strictly between the intersection with $\Sph$ and before the intersection with $M_1\Sph$.  The intersection of $\gamma$ with $M_1\Sph$ can be used to recover the first CF digit of $\gamma_+$, since we have that  $M_1^{-1}(z)=-1/z+a_1$. However,  intersections of $\gamma'$ with other $\Mod$-translates of $\Sph$ do not correspond to digits of $\gamma_+$. We wish to find a subset of $\Sph$ for $\gamma$ to intersect with that does \emph{not} detect these ``spurious'' intersections of $\gamma'$ seen in Figure \ref{fig:building section}a, but does continue to detect (most of) the crossings of $\gamma$ with $M_i\Sph$. In this way intersections of $\gamma$ with $\Mod\Sph$ correspond strongly to iterations of the shift map $T^i$ on $\gamma_+$, and so the behavior of  geodesic flow will strongly correlate to the behavior of the shift map $T$. We will do this in the unit tangent bundle of $\Hyp^2_\R$, by restricting the allowed unit vectors over $\Sph$. The process is summarized in Figure \ref{fig:building section}.

We first quickly prove that $M_1^{-1}\gamma$ in fact crosses $\Sph$. Indeed, we have that $M_1^{-1}\gamma_+$ is inside $\Sph$, while $\norm{a_i}\geq 2$, so that $M_1^{-1}(\gamma_-)=M_1^{-1}(\infty)=-a_i$ is outside of $\Sph$. While this bound appears to deteriorate to $\norm{M_i^{-1}\gamma_-}\geq 1$ with additional iterations, by looking at the permissible digits one shows that $\norm{M_i^{-1}\gamma_-}$ is bounded below by the golden ratio $\phi$. See Remark \ref{rmk:what goes wrong} for the more general approach to this step. 

In Figure \ref{fig:building section}a, we see that $\gamma'$ has (at least) two intersections occur that we do \emph{not} want to code: the intersection with the sphere centered at the point $(1,0)$, and the intersection with the vertical line $x=1.5$. The first of these is avoided simply by restricting to the vectors in $T^1\Sph$ that point towards $K=[-1/2,1/2)$ (i.e., whose corresponding geodesics terminate in $K$). By completeness, any (non-identity) $\Zee$-translate of these vectors must land outside of $K$. In particular, the corresponding vectors on the sphere centered at $(1,0)$ point to the interval $[1/2,3/2)$, whereas we know that $M_1^{-1}\gamma_+\in K$. Thus this  intersection is avoided.

The second spurious intersection in our example requires more work, and we rule it out by making two observations about $M_1^{-1}\gamma'$ that are predicated on the use of horoheight and horoballs (shown in green in Figure \ref{fig:building section}, cf.~Ford circles). We may measure horoheight either from $\infty$, in which case horoheight is simply the $y$ coordinate and a horoball is a  set of the form $\{(x,y):y>y_0\}$, or from a (rational) point on the $x$-axis, in which case horoheight can be thought of as depth into the corresponding cusp, and horoballs appear as Euclidean disks tangent to the $x$-axis. For our first observation, the fact that $\norm{M^{-1}_1\gamma_+}\le 1/2$ and $\norm{M^{-1}_1 \gamma_-} \ge \phi$  implies that the intersection of $M^{-1}_1\gamma'$ with $\Sph$ occurs away from the $x$-axis, in the smaller ``wall'' region (see Figure \ref{fig:building section}b) $$\W=\left\{z\in \Sph \mst \Im(z)>\sqrt{\frac{3}{2}(5\sqrt{5}-11)}\approx 0.52\right\},$$ and that the intersection with $M^{-1}_1\Sph$ is likewise bounded away from the $x$-axis. That is, $M_1^{-1}\gamma'$ is contained in a horoball $\mathcal B=\{y>\epsilon\}$ for some $\epsilon>0$. For our second observation, consider a mapping $M\in \Mod$ that sends the line $x=1.5$ to $\Sph$. Normalizing $M_1^{-1}\gamma'$ further by $M$, we see (Figure \ref{fig:building section}c) that $MM_1^{-1}\gamma' \subset M \mathcal B$ is now contained in one of the horoballs based at a finite rational point. In particular, this provides (see Corollary \ref{cor:h0}) an upper bound $c=\left(\frac{3}{2}(5\sqrt{5}-11)\right)^{-1/2}\approx 1.923$ on how far $MM_1^{-1}\gamma'$ travels away from the $x$-axis. We can reject the intersection of $MM_1^{-1}\gamma'$ with $\Sph$ (and thus the intersection of $M_1^{-1}\gamma'$ with the line $x=1.5$) by restricting to the vectors in $T^1\W$ that are returning from a cusp excursion towards $\infty$ of depth at least $c$. We will denote by $\CW$ the set of such vectors that also point towards $K$. This is our desired refinement of $\Sph$ (Figure \ref{fig:building section}d).

\begin{figure}[ht]
\begin{subfigure}[b]{0.45\textwidth}\includegraphics[width=\textwidth]{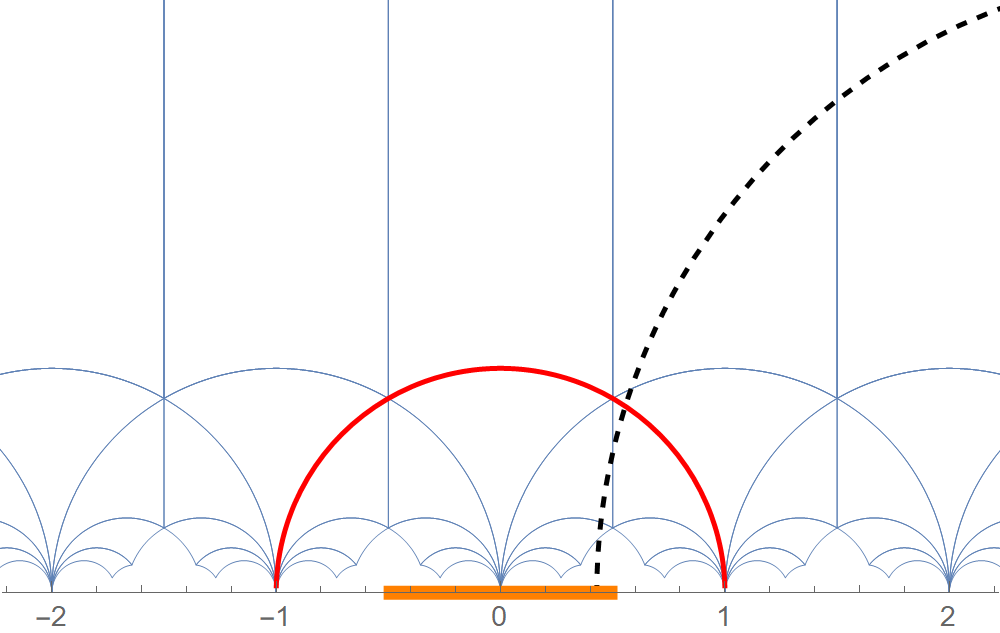}\caption{The geodesic $M_1^{-1}\gamma$ (dashed) intersects the sphere $\Sph$ (red). The segment $M^{-1}_1\gamma'$ has spurious intersections with $\Mod$-translates of $\Sph$, e.g. the line $x=1.5$.}\end{subfigure}
\hspace{.05\textwidth}
\begin{subfigure}[b]{0.45\textwidth}\includegraphics[width=\textwidth]{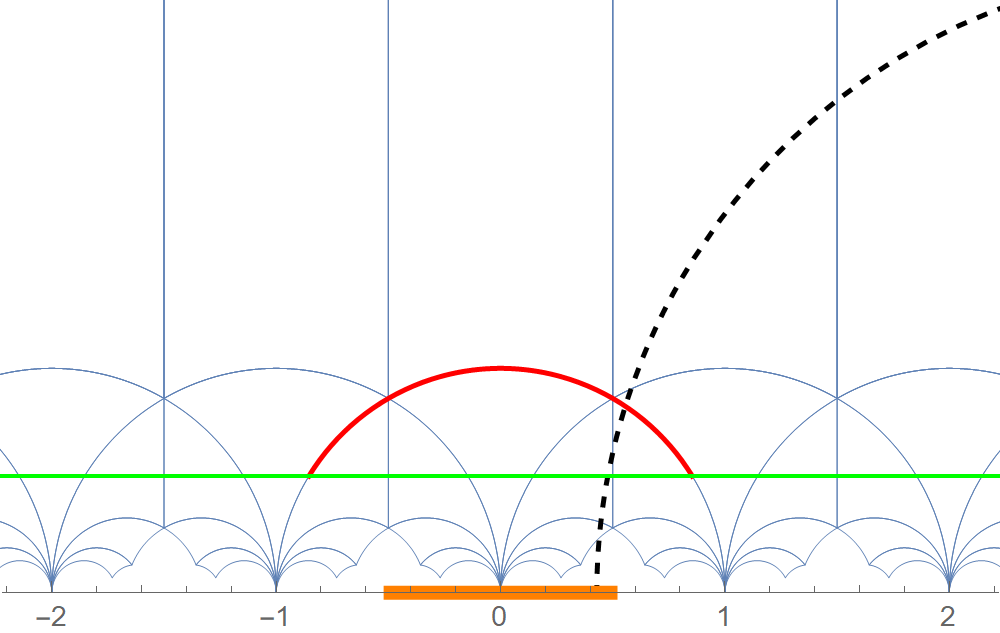}\caption{The intersection occurs in the horoball $y>0.52\ldots$ (green boundary), so we $\Sph$ is replaced with the wall $\W$ (still red)}\end{subfigure}
\begin{subfigure}[b]{0.45\textwidth}\includegraphics[width=\textwidth]{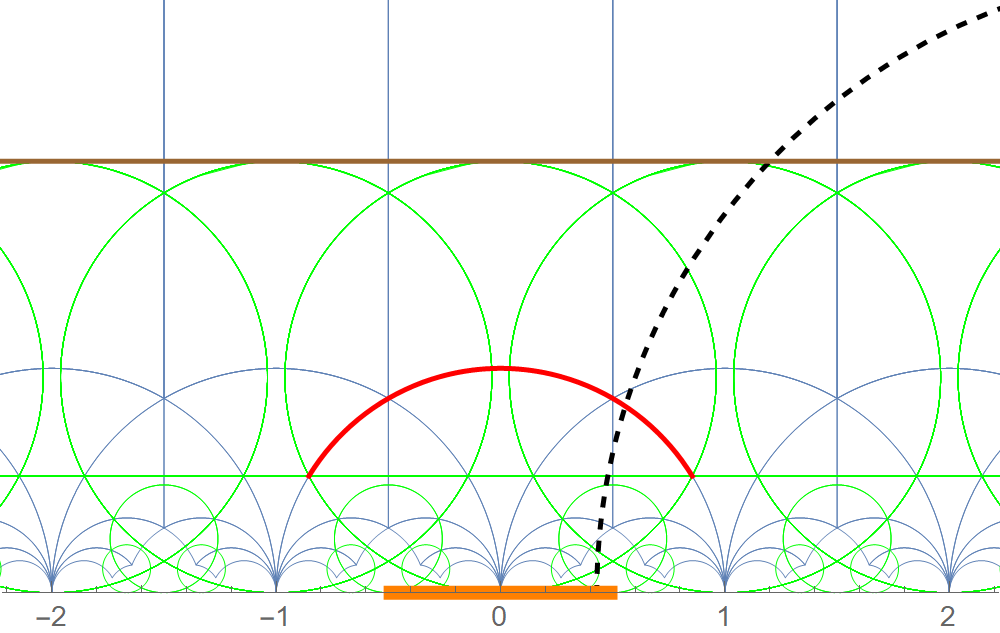}\caption{Translates of the horoball by $\Mod$ to other cusps have horoheight bounded above by  $y=1.923\ldots$ (brown line).}\end{subfigure}
\hspace{.05\textwidth}
\begin{subfigure}[b]{0.45\textwidth}\includegraphics[width=\textwidth]{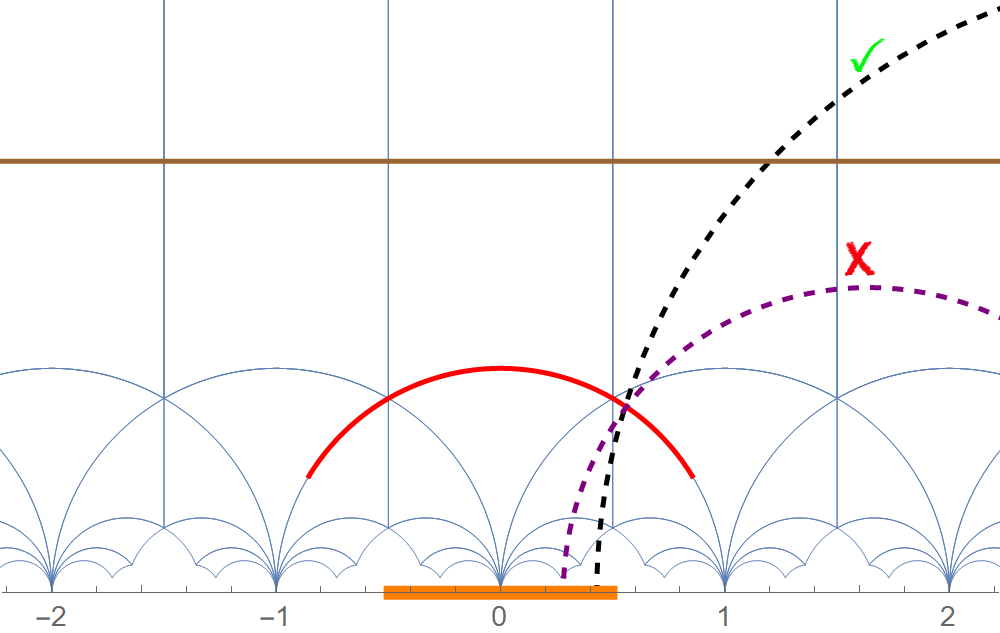}\caption{Any geodesics that have small horoheight are excluded from the unit tangent bundle of $\W$ to produce $\CW$.}\end{subfigure}
\caption{Constructing the section $\CW$ of geodesic flow.}
\label{fig:building section}
\end{figure}

Our choice of $\CW$ by construction avoids all spurious intersections, but may also inadvertently ignore some intersections corresponding to small digits or even entire geodesics. We say that a geodesic $\gamma$ is \emph{markable} if it intersects $\Mod$-translates of $\CW$ infinitely often in both the past and future, and it is easy to see that almost every geodesic is markable, see Corollary \ref{cor:markablegeneric}.  The Markable Geodesic Theorem \ref{thm:markable} records the desired link between the CF digits of a markable geodesic $\gamma$ and its intersections with $\Mod$-translates of $\CW$: intersections occur only with walls of the form $M_i\CW$, the intersections occur in the desired order, and no other intersections occur.

With the section $\CW$ in hand, we return to the question of ergodicity. We begin by working through a number of closely related functions acting on different spaces, pulling ergodicity from one function to the next. The ergodicity of geodesic flow on the modular surface $\Mod\backslash \Hyp^2_\R$ implies the ergodicity of the first return map to the projection of $\CW$ onto the modular surface. We can then lift this first return map back to $\Hyp^2_\R$ to obtain an isomorphic and thus equally ergodic map $\psi:\CW\to\CW$ (Proposition \ref{prop:phiergodic}). We then conjugate this system with the projection $\pi$ from the unit tangent bundle of $\Hyp^2_\R$ to $\hat \R\times \hat \R$ that takes any geodesic $\gamma$ to its forward and backward endpoints $(\gamma_+,\gamma_-)$, obtaining an ergodic mapping $\Psi=\pi\circ\psi\circ \pi^{-1}$ on $\pi(\CW)$. This map acts by taking a point $(\gamma_+, \gamma_-)$ to $(M_i^{-1} \gamma_+,M_i^{-1} \gamma_-)$ for some $i$. Thus, in the first coordinate, $\Psi$ acts by $T^i$, where $i$ may depend on the value of $\gamma_+$, i.e., this is a jump transformation associated to $T$.

Although we could conclude this jump transformation is ergodic, the ergodicity of a jump transformation does not imply the ergodicity of the original map in general. So to recover the ergodicity of $T$, we step back to $\hat \R\times \hat \R$. Namely, we consider a natural-extension-like function $\hat{T}$ on $K\times \hat\R$ such that the action of $\hat{T}$ on the first coordinate is simply $T$. We show that the action of $\hat{T}$ on $\overline{K}=\bigcup_{i=0}^\infty \hat{T}^i \pi(\CW)$ is well-behaved (Lemma \ref{lemma:Tinvariance}) and that, in fact, $\Psi$ is simply the map induced by restricting $\hat{T}$ to $\pi(\CW)$ (Lemma \ref{lemma:induced}). Induced maps have far greater structure than jump transformations and so we are able to conclude the ergodicity of $\hat{T}$ on $\overline{K}$ from the ergodicity of $\Psi$ (Lemma \ref{lemma:kbarergodic}) and from there conclude the ergodicity of $T$ by restricting to the first coordinate.

\begin{remark}
\label{rmk:what goes wrong}
There are two sources of complexity in the full proof of Theorem \ref{thm:main}. 
The first is that we would like to work in sufficient generality to include Heisenberg continued fractions. This requires working with hyperbolic spaces defined over complex numbers and quaternions, and obtaining some new results about inversions for the corresponding horospherical coordinates with boundary, see Theorem \ref{thm:InversionIdentities}.
The second source of complexity is the fact that, even for simple CF algorithms the point $M_i^{-1}\gamma_-$ need not remain outside of the sphere $\Sph$, and the properness assumption is necessary to guarantee that some intersections do occur. For example, the $\alpha$-CF algorithm with $\alpha >2/3$ and inversion $\iota_+$ would send the geodesic with endpoints $\gamma_+=\frac{2}{3}$ and $\gamma_-=\infty$ to  the geodesic with endpoints $M_1\gamma_+=1/2$ and $M_1\gamma_-=-1$, which does not intersect $\Sph$. However, $M_i\gamma$ cannot remain under $\Sph$ forever: applying the identity $\norm{1/x-1/y}\norm{x}\norm{y}=\norm{x-y}$, we see that additional iterations of the shift map must pull the endpoints of $M_i \gamma$ apart and  push $M_i \gamma_-$ out of the unit circle. 
The need to wait several iterations before a collision is detected then complicates the construction of the wall region $\W$ and the section $\CW$.
\end{remark}

\subsection{Further Remarks}
\label{sec:further}

Iwasawa CFs are the most general setting for our methods, which rely heavily on the fact that Iwasawa inversion spaces are boundaries of rank one symmetric spaces of non-compact type. Indeed, Iwasawa inversion spaces are precisely the spaces with this property, with the exclusion, due to the break down of vector-space-based techniques, of the exceptional $\X^1_{\mathbb O}$ that can be defined over the non-associative octonions.  Our notion of \emph{Iwasawa inversion space} differs slightly from the notion of \emph{Iwasawa groups} of \cite{MR4126256}, which excludes $\X^n_\R$ and allows $\X^1_{\mathbb O}$.

We remark further that boundaries of rank one symmetric spaces of non-compact type are arguably the most general setting for geometric CFs and Diophantine theory:  they are characterized \cite{MR3283670, COWLING19911} as homogeneous geodesic locally compact spaces admitting both a dilation (a notion of fraction) and a well-behaved inversion. (The Cygan metric we work with is not itself geodesic, but gives rise to a geodesic path metric.)

The present work suggests the following further directions of study:

\begin{question}
Under what conditions is the invariant measure for the CF shift map finite or (piecewise) analytic? 
\end{question}

\begin{question}
Is the CF shift map mixing?
\end{question}

\begin{question}
Does Theorem \ref{thm:main} hold for incomplete Iwasawa CFs, or for improper Iwasawa CFs with weak contact with the unit sphere (such as J.~Hurwitz CFs)?
\end{question}

\begin{question}
Can one characterize periodic Iwasawa CF expansions, analogously to the quadratic surd characterization of periodic regular CFs in $\R$ (cf.~\cite{Vperiodic})?
\end{question}

\begin{question}
Can one describe the set of exceptions to tail-equivalence in Theorem \ref{thm:IntroTail} (cf.~\cite{Lakein})?
\end{question}

\begin{question}
What Iwasawa CF algorithms are not represented in Table \ref{tab:examples}?
\end{question}

\subsection{Outline of the paper}
Following this introduction, in \S \ref{sec:defn} we provide the general theory and definitions for Iwasawa inversion spaces. In \S \ref{sec:IwasawaCF} we define Iwasawa CFs, give further examples (including Table \ref{tab:examples}) and study conditions that guarantee discreteness, properness, and completeness. In \S \ref{sec:convergence}, we quickly prove the convergence of Iwasawa CFs. In \S \ref{sec:markable}, we will build up the theory surrounding markable geodesics, culminating in the Markable Geodesic Theorem. In \S \ref{sec:ergodicity}, we use the Markable Geodesic Theorem to prove the ergodicity of the CF shift map for an Iwasawa CF expansion and, in applications of this result, prove Theorems \ref{thm:notmain} and \ref{thm:IntroTail}.

\subsection{Acknowledgements}
A.L. was supported by University of Michigan NSF RTG grant 1045119. This article was written during visits by the authors to University of Texas at Tyler, George Mason University, University of Michigan, and the Ohio State University. The authors thank these institutions for their hospitality, and Simons Travel Grant and GEAR Grant NSF DMS 11-07452 for the travel funding. The authors would also like to thank Jayadev Athreya  and Ralf Spatzier for their helpful comments, and the anonymous referee for the extensive suggestions that improved the exposition of this paper.

\section{General Theory}\label{sec:defn}
We now outline the structure of Iwasawa Inversion spaces $\X=\X^n_k$, the associated upper half-spaces $\Hyp^{n+1}_k$, and the continued fraction algorithms that can be built on $\X$ using this structure. We encourage the reader to skip this section on the first reading, following the intuition of the Euclidean space $\X=\X^n_\R=\R^n$ and hyperbolic half-space $\Hyp=\Hyp^{n+1}_\R$ lying above it.

\subsection{Iwasawa Inversion Spaces}
\label{sec:space}
Abstractly, an Iwasawa inversion space $\X$ is an Iwasawa $N$-group associated by the Iwasawa (KAN) decomposition to a non-exceptional rank one semi-simple Lie group $G$ and the parabolic boundary at infinity of the rank one symmetric space $G/K$. We now recall the explicit construction and Euclidean-like structure of these spaces. Most of the contents of this section can be found in \cite{MR4126256,MR1695450,MR1146815}.

Fix an associative division algebra $k$ over the reals --- the real, complex, or quaternionic numbers\footnote{When working over quaternions, we will use the convention $\frac{p}{q}:=pq^{-1}$.}  --- and an integer $n\geq 1$. (It appears that one could also consider the exceptional case of octonions, but we will not do so here.)  Recall that $k$ has a real part $\Re(k)$ isomorphic to $\R$ and a complementary imaginary part $\Im(k)$ satisfying $\dim_\R(\Im(k))=\dim_\R(k)-1$. We denote the standard norm of an element of $k$ or $k^n$ by $\Norm{\cdot}$, and refer to $\Norm{\cdot}$-preserving $k$-linear automorphisms of $k^n$ as \emph{unitary} transformations.

\begin{remark}For $k=\R$, one has $\Im(k)=\{0\}$. Note that $\Im(k)$ remains a subset of $k$; in particular, we do not identify $\Im(k)$ with $\R$ when $k=\C$. We furthermore exclude nonholomorphic transformations such as $z\mapsto \overline z$ from the unitary group, purely for notational convenience.
\end{remark}

\begin{defi}[Iwasawa Inversion Space]
The \emph{Iwasawa inversion space} $\X=\X_k^n$ is the set $k^n\times \Im(k)$ with coordinates $(z,t)$ and group law
\(
(z,t)*(z',t')=(z+z', t+t'+2\Im \langle z, z'\rangle),
\)
where the inner product of the vectors $z, z'$ is given by $\langle z, z'\rangle=\sum_i \overline{z_i}z'_i$. 
\end{defi}

Over the reals, $\X^n_\R$ reduces to $\R^n$ with $*$ acting by the usual vector addition. For $k\neq \R$, $\X^n_k$ is a step-2 nilpotent group (one uses $*$ to emphasize the non-commutativity), with identity $(0,0)$, and the inverse of a group element $(z,t)$ given by $(-z,-t)$. 

One gives $\X$ a gauge $\norm{\cdot}$ and Cygan metric $d$ (also known in different contexts as the Kor\'anyi metric or gauge metric) by defining
\begin{linenomath}\begin{equation*}\norm{(z,t)}:=\Norm{\Norm{z}^2+t}^{1/2}, \hspace{.35in} d((z,t), (z',t')):=\norm{(-z,-t)*(z,t)}.
\end{equation*}\end{linenomath}

The Cygan metric is largely analogous to the Euclidean metric, insofar as its automorphisms include analogs of translations (left multiplication by an element of $\X$ is an isometric isomorphism); dilations (for each $r>0$, the mapping $\delta_r(z,t)=(rz,r^2t)$ is a group isomorphism that rescales the metric by factor $r$); and rotations (unitary automorphisms of $k^n$ extend to isometric group isomorphisms of $\X$). 

On the other hand, the metric is fractal for $k\neq \R$: it is not a path metric (cf.~the closely associated Carnot-Carath\'eodory path metric) and gives $\X$ Hausdorff dimension $n\dim_\R(k)+2(\dim_\R(k)-1)$ which is not equal to its topological dimension $(n+1)\dim_\R(k)-1$. The latter is due to the fact that large metric balls are stretched by $\delta_r$ along the $t$ direction, while small ones are flattened out along the $z$ direction.

The \emph{Kor\'anyi} inversion $\iota_-: \X\setminus\{0\}\rightarrow \X\setminus\{0\}$ is defined by
\(\iota_-(z,t)=\left(\frac{-z}{\Norm{z}^2+ t}, \frac{-t}{\Norm{z}^4+\Norm{t}^2}\right).\)

The Kor\'anyi inversion is a natural generalization of the mapping $x\mapsto -1/x$, and in particular satisfies the following pair of identities for $h,h'\in \X\setminus\{0\}$, \cite{COWLING19911}:
\[
\label{eq:inversionidentitiesBasic}
\norm{\iota_- h}=\frac{1}{\norm{h}}, \hspace{.5in} d(\iota_- h, \iota_- h')=\frac{d(h,h')}{\norm{h}\norm{h'}}.
\]

In particular, $\iota_-$ sends each sphere $S(0,r)$ to the sphere $S(0,1/r)$, and preserves the unit sphere. We prove the identities in a broader context in Theorem \ref{thm:InversionIdentities}.

More generally, $\X$ admits inversions of the form 
\(\iota(z,t)=\left(\frac{-A(z)}{\Norm{z}^2+ t}, \frac{-(\det A) t}{\Norm{z}^4+\Norm{t}^2}\right),\)
where $A$ is a unitary transformation of $k^n$. We show in Lemma \ref{lemma:inversion} that all inversions satisfy generalizations of Equations \ref{eq:inversionidentitiesBasic}.

\subsection{Upper Half-Space}
\label{sec:hyp}
Fix an Iwasawa inversion space $\X=\X^n_k$. We extend the structure and Cygan metric of $\X$ to $k^{n+1}$ as follows, motivated by Parker \cite{MR1146815}:

\begin{defi}\label{def:extendcygan}
Extend the Heisenberg group law to $k^n\times k=k^{n+1}$ as 
\(
(z,w)*(z',w')=(z+z', w+w'+2\Im \langle z, z'\rangle),
\)
and the gauge and metric as:
\begin{linenomath}\begin{equation*}\norm{(z,w)}=\Norm{\Norm{z}^2+\Norm{\Re(w)}+\Im(w)}^{1/2}, \hspace{.15in} d((z,t), (z',t')):=\norm{(-z,-t)*(z,t)}.
\end{equation*}\end{linenomath}
\end{defi}

\begin{remark}
In the case $k=\R$, the Heisenberg group law on $k^{n+1}$ reduces to $(z,w)*(z',w')=(z+z', w+w')$, and the gauge reduces to the Euclidean-like $\norm{(z,w)}=\left(\Norm{z}^2+\Norm{w}\right)^{1/2}$. 
One could adjust Definition \ref{def:extendcygan}, by taking a square root along the $\Re(w)$ direction, so that it agrees with the Euclidean metric in the real case. We will not do so.
\end{remark} 

\begin{defi}
 The \emph{upper half-space}  $\Hyp^{n+1}_k\subset k^{n+1}$ is the set 
 \(\Hyp^{n+1}_k = \{ (z,w)\in k^n\times k \st \Re(w)>0\},\) satisfying $\partial \Hyp=\X$.
 
 One gives $\Hyp$ two natural metrics: the restriction of the Cygan metric $d$ on $k^{n+1}$ (this was introduced by Parker in \cite{
MR1146815} for $\Hyp^2_\C$ and generalized by Cao-Parker to $\Hyp^2_{\mathbb O}$ in \cite {MR3764717}); and the negatively-curved hyperbolic metric $d_\Hyp$, defined via an embedding into $\mathbb{P}(k^{n+2})$. 
Unless otherwise noted, $\Hyp$ will always be equipped with the metric $d_\Hyp$.
\end{defi}

\begin{defi}[Projective Embedding]
\label{defi:ProjectiveEmbedding}Let $\phi: k^{n+1}\rightarrow k^{n+2}$ be given by $\phi(z,w)=(1,\sqrt{2}z, w+\Norm{z}^2)$, and set $\Phi=\mathbb{P}\circ \phi: k^{n+1}\rightarrow \mathbb{P}(k^{n+2})$.
\end{defi}

Consider the Hermitian form $\langle \cdot, \cdot \rangle_J$ of signature $(n+1,1)$ defined on $k^{n+2}$ by
\(
J=\begin{bmatrix}
    0 & 0_n & -1\\
    0_n & \operatorname{id}_n & 0_n\\
    -1 & 0_n & 0
\end{bmatrix},
\)
and let $\mathcal S=\{(1:a:b)\st \Norm{a}<2\Re(b)\} \subset \mathbb P(k^{n+2})$ be the Siegel region. One can show that $\Phi$ induces a bijection between $\Hyp$ and $\mathcal S$, and furthermore $\mathcal S$ is the projectivization of the negative cone of $J$. This induces an action of the projective unitary group $G=\mathbb{P}U(J)$ on $\Hyp$, cf.~\S \ref{sec:convergence}.

\begin{defi}(Hyperbolic metric)
The hyperbolic metric $d_\Hyp$ on $\Hyp$ is the unique $G$-invariant Riemannian metric on $\Hyp$ with sectional curvature pinched in the range $[-1, -1/4]$ if $k\neq \R$ or equal to $-1$ if $k=\R$.
\end{defi}

For $\Hyp=\Hyp^2_\R$, $d_\Hyp$ is agrees with the familiar metric $\frac{1}{y}ds$ if one takes $x=z$ and $y=w^2$. One has $\Phi(\Hyp)=\{(1:a:b) \st 2b>a^2\}\subset \mathbb{RP}^2$, and a projective change of coordinates recovers the Klein disk model of $\Hyp^2_\R$ with its $SO(2,1)$-invariant metric. 

In general, the Siegel region is projectively equivalent to a unit ball in projective space $\mathbb P(k^{n+2})$. The mapping $\Phi\vert_\X:\partial \Hyp \rightarrow \partial \Phi(\Hyp)$ omits a single point, which we identify with the point $\infty$ in the one-point compactification of  $k^{n+1}$ (and its subsets $\X$ and $\overline{\Hyp}$).

\subsection{Inversion Theorem}
Returning to the Cygan metric, we record two connections to the projective embedding:
\begin{lemma}[Parker \cite{MR1146815}]
\label{lemma:Jform}
Suppose $p,q\in \overline{\Hyp}$, with either $p$ or $q$ in $\X=\partial \Hyp$. Then the Cygan metric satisfies $d(p,q)=\Norm{\langle \phi(p),\phi(q)\rangle_J}^{1/2}$.
\end{lemma}

\begin{lemma}\label{lemma:thirdcoordinate}
Let $h\in \overline{\Hyp}$ and denote $\phi(h)=(1,a,b)$. Then $\norm{h}=\Norm{b}^{1/2}$.
\begin{proof}
This is immediate from Definitions \ref{def:extendcygan} and \ref{defi:ProjectiveEmbedding}.
\end{proof}
\end{lemma}

With the above machinery, we can provide a simple description of the Kor\'anyi inversion, extended to $\overline \Hyp$, and prove the inversion identities \eqref{eq:inversionidentitiesBasic}.
\begin{lemma}
\label{lemma:inversion}
The Kor\'anyi inversion $\iota_-: \overline{\Hyp}\setminus\{0\}\rightarrow \overline{\Hyp}\setminus\{0\}$ given by the mapping
\((z,w)\mapsto \left(\frac{-z}{\Norm{z}^2+w}, \frac{\overline{w}}{\Norm{\Norm{z}^2+w}^2}\right)\)
is induced by the matrix $J\in G$. That is, setting $\phi(z,w)=(1,a,b)$, one has $\phi(\iota_-(z,w))=(1,-a/b,1/b)=\frac{\phi(z,w)}{-b}$, and in $\mathbb{P}(k^{n+2})$ one has $\Phi(\iota_-(z,w))=J\Phi(z,w)$.
\begin{proof}
We have $\phi(z,w)=(1,\sqrt{2}z, \norm{z}^2+w),$ so that $J\phi(z,w)=( -(\norm{z}^2+w),\sqrt{2}z,-1)$. Up to a factor of ${-(\Norm{z}^2+w)}$, this is equivalent to \begin{linenomath}\begin{align*}
    \left(1,\sqrt{2}\frac{-z}{\Norm{z}^2+w},\frac{1}{\Norm{z}^2+w}\right)
    &=\left(1,\sqrt{2}\frac{-z}{\Norm{z}^2+w},\Norm{\frac{-z}{\Norm{z}^2+w}}^2+\frac{\overline{w}}{\Norm{\Norm{z}^2+w}^2}\right),
\end{align*}\end{linenomath}
which in turn is equal to $\phi(\iota_-(z,w))$ as desired.
\end{proof}
\end{lemma}

\begin{thm}[Inversion Theorem]
\label{thm:InversionIdentities}Let $h\in (\Hyp\cup \X)\setminus\{0\}$ and $h'\in \X\setminus\{0\}$. The following identities hold for the Kor\'anyi inversion $\iota_-$, Cygan metric $d$, and gauge $\norm{\cdot}$:
\[\label{eq:inversionidentitiesAdvanced}\norm{\iota_- h}=\frac{1}{\norm{h}} \hspace{.25in}\text{and} \hspace{.25in} d(\iota_- h, \iota_- h')=\frac{d(h,h')}{\norm{h}\norm{h'}}.
\]
\begin{proof}
Write $\phi(h)=(1,a,b)$ and $\phi(h')=(1,a',b')$.  By Lemma \ref{lemma:inversion}, $\phi(\iota_-(h))=(1,-a/b,1/b)$, and the first identity thus follows from Lemma \ref{lemma:thirdcoordinate}.

Since $h'\in \X$, Lemma \ref{lemma:Jform} gives $d(h,h')=\Norm{\left\langle \phi(h),\phi(h')\right\rangle_J}^{1/2}$ and $d(\iota_- h, \iota_- h')=\Norm{\left\langle \iota_- \phi(h),\iota_- \phi(h')\right\rangle_J}^{1/2}$. Using Lemmas \ref{lemma:inversion} and \ref{lemma:thirdcoordinate}, we obtain:
\begin{linenomath}\begin{align*}
d(\iota_- h, \iota_- h')=
        \Norm{\left\langle\frac{\phi(h)}{-b},\frac{\phi(h')}{-b'}\right\rangle_J}^{1/2}=
        \frac{
            d(h,h')}
            {\Norm{b}^{1/2} \Norm{b'}^{1/2}}
        = \frac{d(h,h')}{\norm{h}\norm{h'}},
\end{align*}\end{linenomath}
providing the second identity.
\end{proof}
\end{thm}

\begin{remark}
Surprisingly, Lemma \ref{lemma:Jform} and the second identity of Theorem \ref{thm:InversionIdentities} fail when both $h$ and $h'$ lie in $\Hyp$.
\end{remark}

Compositions of diagonal elements of $G$ (as well as certain conjugation actions) with the Kor\'anyi inversion continue to satisfy the conclusions of Theorem \ref{thm:InversionIdentities}. We define:
\begin{defi}
\label{defi:inversion}
An \emph{inversion} is a M\"obius transformation $\iota: \X\setminus\{0\}\rightarrow \X\setminus\{0\}$ satisfying the conclusions of Theorem \ref{thm:InversionIdentities}. \end{defi}

It follows from the classification of isometries of $\Hyp$ that every inversion factors as a composition of a rotation and the Kor\'anyi inversion.
\begin{lemma}
\label{lemma:inversionform}
If $\iota$ is an inversion, then there exists a unitary mapping $f: k^n\rightarrow k^n$ such that $\iota=f\circ \iota_-$.
\begin{proof}
Since $\iota$ is a M\"obius transformation, it extends to an isometry of $\Hyp$. The mapping $f=\iota \iota_-$ is an isometry of $\Hyp$ that fixes the points $0$ and $\infty$. It therefore maps the geodesic $\gamma$ joining $0$ and $\infty$ to itself. Since $\iota_-$ and $\iota$ fix the point $(0,1)\in \gamma$ by the first part of \eqref{eq:inversionidentitiesAdvanced}, the same must be true for $f$. Thus, $\iota$ is represented in $U(J)$ by a matrix of the form 
\[\label{eq:inversionmatrix}
\begin{bmatrix}
    0 & 0_n & -1\\
    0_n & A & 0_n\\
    -1 & 0_n & 0
\end{bmatrix},
\]
where $A$ is a unitary matrix over $k^n$.
\end{proof}
\end{lemma}

In addition to the (negative) Kor\'anyi inversion $\iota_-$, we will also be interested in the positive inversion $\iota_+$ corresponding to the matrix $A=-I_n$ in \eqref{eq:inversionmatrix}, and the conjugation inversion $\iota_c$ corresponding to the diagonal matrix $A$ with diagonal entries $(-1, 1, 1, \ldots, 1)$. For example, for $p=(x,y,z)\in \R^3$, one has $\iota_-(p)=-p/\Norm{p}^2$, $\iota_+(p)=p/\Norm{p}^2$, and $\iota_c(p)=(x,-y,-z)/\Norm{p}^2$. Note that under the standard identification of $\C$ with $\R^2$, the mapping $z\mapsto 1/z$ corresponds to the inversion $\iota_c$.

\subsection{Isometries, Lattices, and Fundamental Domains}

We thus have an Iwasawa inversion space $\X$ and associated hyperbolic space $\Hyp$, with the unitary group $G$ acting on $\Hyp$ by isometries with respect to the Riemannian metric $d_\Hyp$, and by generalized M\"obius transformations on $\widehat X = \X\cup\{\infty\}$. One shows that $G$ is in fact the holomorphic isometry group of $\Hyp$, and the group of (1-quasi-)conformal mappings of $\X\cup\{\infty\}$. Restricting $G$ to the set of transformations $\Stab_G(\infty)$ preserving infinity provides an action on $\X$ that can be identified with the group of similarities of $\X$. This allows us to think of $\Isom(\X)$ as a subgroup of $\Isom(\Hyp)$. 

The group $G$ is, in fact, a rank-one simple Lie group, with an Iwasawa decomposition $G=KAN$. One can identify the subgroup $N$ with the space $\X$ (with the group structure provided above), and the subgroup $A$ with the group of dilations $\{\delta_r\st r>0\}$. The subgroup $K$ can be identified with the stabilizer of the point $(0,1)\in \Hyp$, and includes the Kor\'anyi inversion.

We will be interested in lattices and fundamental domains in $\Isom(X)$ and $\Isom(\Hyp)$, equipped with the respective Haar measures. 
\begin{defi}\label{defi:fundamentaldomain}
Let $Y$ be a metric space with an $\Isom(Y)$-invariant measure. A \emph{lattice} is a discrete subgroup $\Gamma\subset \Isom(Y)$ such that the quotient $\Gamma \backslash \Isom(Y)$ has finite measure. The lattice is \emph{uniform} if $\Gamma \backslash \Isom(Y)$ is furthermore compact, and non-uniform otherwise.

A \emph{fundamental domain} for $\Gamma$ is a measurable set $K\subset Y$ such that $\X=\bigcup_{a\in \Gamma} aK$ and the overlap $K\cap \bigcup_{a(\neq \id)\in \Gamma} aK$ has measure $0$. 

A \emph{rounding mapping} $[\cdot]:Y\rightarrow \Gamma$ associated to $\Gamma$ and $K$ is defined, almost everywhere, by the property that for each $a\in \Gamma$ and $x\in K$, one has $[a(x)]=a$. 
This property defines $[\cdot]$ uniquely away from the overlap, and $[\cdot]$ provides some choice of admissible values has been made for points in the overlap.
\end{defi}

\section{Iwasawa Continued Fractions}\label{sec:IwasawaCF}
We can now define Iwasawa continued fractions and establish some auxiliary terminology and notation.

\begin{defi}(Iwasawa Continued Fraction)
The Iwasawa Continued Fraction Algorithm is defined by the following data:
\begin{enumerate}
    \item An associative division algebra $k$ over $\R$ and integer $n\geq 1$,
    \item The associated Iwasawa inversion space $\X=\X^n_k$,
    \item An inversion $\iota$ (see Definition \ref{defi:inversion}),
    \item A lattice $\Zee \subset \Isom(\X)$, a fundamental domain $K\subset \X$ for $\Zee$, and an associated  rounding mapping $[\cdot]: \X\rightarrow \Zee$ (see Definition \ref{defi:fundamentaldomain}).
\end{enumerate}
Associated to an Iwasawa CF algorithm, we have:
\begin{enumerate}	
\setcounter{enumi}{4}
	\item The hyperbolic space $\Hyp=\Hyp^{n+1}_k$ satisfying $\partial \Hyp=\X$,
	\item The holomorphic isometry group $G$ of $\Hyp$,
	\item The modular group $\Mod=\langle \iota, \Zee\rangle \subset G$.
	\item The shift map $T:K\rightarrow K$ defined  by $T(0)=0$ if $0\in K$ and otherwise by \(T(x)=[\iota(x)]^{-1}(\iota(x)).\) 
\end{enumerate}
\end{defi}

For a point $x\in \X$, we can then inductively define the continued fraction digits $a_i\in \Zee$ and forward iterates $x_i\in K$ by taking
\begin{linenomath}\begin{align*}
    &a_0 = [{x}], &&x_0=a_0^{-1}(x),\\
    &a_{i+1} = [\iota(x_i)], &&x_{i+1}=a_{i+1}^{-1}(\iota(x_i))=T(x_i),
\end{align*}\end{linenomath}
where the sequences terminate if at some point $x_i=0$. The (possibly finite) sequence $(a_i)$ of elements of $\Zee$ is the \emph{continued fraction sequence} of $x$. (Note that later in the paper, we will assign a bi-infinite string of digits to pairs of points one of which is in $K$, resulting in a different notion of $a_0$. For this reason, for points in $K$ we will leave $a_0$ undefined.)

Given a sequence $(a_i)$ of elements of $\Zee$ (possibly arising from the above algorithm), one defines the \emph{convergent mappings} $M_i\in \Mod$ inductively by setting $M_0$ to be the identity mapping and $M_{n+1}= M_n\circ \iota^{-1} \circ a_{n+1}$. (In the sequel, we will often suppress the $\circ$ notation for convenience.) By construction, we see that $x_0=M_n (x_n)$. For each $i$, the $i^{th}$ \emph{convergent} of the continued fraction is then the point $M_i(0)$. Note that $T^i x_0=x_i=M_i^{-1}(x_0)$.

We will be interested in conditions on the continued fraction algorithm that guarantee the following properties:
\begin{defi}
The continued fraction algorithm is \emph{convergent} if the continued fraction digits of almost every point $x\in K$ produce convergents $M_i(0)$ that indeed converge to $x$ (clearly, every finite expansion is convergent). The algorithm is \emph{ergodic} if the shift map $T$ is ergodic.
\end{defi}

We will use the following definition of ergodicity:
\begin{defi}
Let $(A,\mu)$ be a measure space and $f:A\rightarrow A$ a measurable (but not necessarily measure-preserving) transformation. Then, $f$ is said to be ergodic with respect to $\mu$ if for every measurable $B\subset A$, $\mu(f^{-1}B\triangle B)=0$ implies that $\mu(B)=0$ or $\mu(A\setminus B)=0$. If $\phi:A\rightarrow A$ is a measurable flow, then $\phi$ is ergodic with respect to $\mu$ if for every measurable $B\subset A$,  $\mu(\phi_t(B)\triangle B)=0$ for all $t\in\mathbb{R}$ implies that $\mu(B)=0$ or $\mu(A\setminus B)=0$.
\end{defi}

\begin{remark}
Note that with this definition, ergodicity with respect to a measure $\mu$ implies ergodicity with respect to any measure that is equivalent to $\mu$. In this paper, the relevant measure (or class of equivalent measures) will always be clear from context, and will often be a Lebesgue or Haar measure.
\end{remark}

We will prove the convergence of the Iwasawa CFs under the assumptions of \emph{properness} and \emph{discreteness}:
\begin{defi}[Properness and Discreteness]
The Iwasawa continued fraction is \emph{proper} if the closure of $K$ is bounded away from the unit sphere: $\rad(K)=\sup\{\norm{x}\st x\in K\}<1$. It is \emph{discrete} if $\Mod$ is a discrete subgroup (and therefore, by construction, a lattice) in $G$.
\end{defi}

There do exist convergent Iwasawa continued fractions that are not proper, most notably regular continued fractions on $\R$ and J.~Hurwitz continued fractions on $\C$. Likewise, one can construct proper but non-discrete Iwasawa continued fractions: for example, let $\X=\R$, $\Zee=\epsilon \Z$, and $K=(-\epsilon/2,\epsilon/2]$. The resulting continued fraction is generally not discrete, but will be convergent by the \'Sleszy\'nski-Pringsheim Theorem \cite{sleshinskiy} for $\epsilon<1/2$.

To prove ergodicity, we will need a further assumption of \emph{completeness}, which rules out hidden symmetries:
\begin{defi}[Completeness]
The Iwasawa continued fraction is \emph{complete} if one has $\Stab_\Mod(\infty)=\Zee$.

For an incomplete continued fraction, one may pass to the \emph{completion} by replacing $\Zee$ with the lattice $\Stab_\Mod(\infty)$ and making a corresponding modification to the fundamental domain $K$ and rounding function $[\cdot]$. This will result in what are often termed ``folded'' variants (see \S \ref{subsubsec:folded}).
\end{defi}

\begin{defi}
\label{defi:centrallySymmetric}
The Iwasawa continued fraction is \emph{incomplete with $n$ central symmetries} if there exists a set $\mathcal{R}\subset\Isom(\X)$ such that
\begin{enumerate}
    \item Every element of $\mathcal{R}$ fixes $0$, i.e., is a rotation around the origin,
    \item The only element of $\Zee$ to fix $0$ is the identity,
    \item $\Stab_\Mod(\infty)=\langle \Zee,\mathcal{R}\rangle$,
    \item Every element of $\Stab_\Mod(\infty)$ can be written uniquely  as $ra$ for some $r\in \mathcal{R}$, $a\in \Zee$, and uniquely as $a'r'$ for some $a'\in \Zee$, $r'\in \mathcal{R}$, and,
    \item $\mathcal R$ contains $n$ elements.
\end{enumerate}
The set $\mathcal{R}$ is said to be the set of central symmetries of $\Mod$. We say that the fundamental domain $K$ for $\Zee$ is symmetric if for any $r\in \mathcal{R}$, $rK$ is $K$ up to a set of measure zero.
\end{defi}

\subsection{Further Examples}\label{sec:examples}~
With all of our notation now in place, we may describe many examples of Iwasawa continued fractions. In Table \ref{tab:examples}, we list several types of continued fractions, and for each of them denote the Iwasawa inversion space $\X$ on which it exists; the lattice $\Zee$, which will often act by left-translation by a subset of $\X$; the fundamental domain $K$; the inversion, which in all cases will be identified by a $\iota$ signature; whether it is complete and proper (the columns C and P respectively); and some basic references.

\begin{table}
\caption{Examples of Iwasawa continued fractions. The examples in $\X^2_\R=\R^2$ are usually presented as complex CFs. See \S\ref{subsec:landscape} and \S \ref{sec:examples} for more information about the algorithms.} \label{tab:examples}

\begin{tabular}{| p{2.3cm} | c | p{2cm} | p{2.2cm} | c | c | c | p{1.5cm} | }
\hline
     Name:& $\X$ & $\Zee$ & $K$ & $\iota$ & C&P  & References  \\
     \hline
     \hline
     Regular & $\X_{\R}^1$ & $\Z$ & $[0,1)$ & $\iota_+$ & N&N & \cite{Series1} \\
     \hline
     Backwards & $\X_{\R}^1$ & $\Z$ & $[0,1)$ & $\iota_-$ & Y&N & see \S \ref{subsec:backwardsCF} \\
     \hline
     Nearest \mbox{Integer} & $\X_{\R}^1$ & $\Z$ & $\left[-\frac{1}{2},\frac{1}{2}\right)$ & $\iota_+$ & N&Y & \cite{IT} \\ 
     \hline
        Nearest \mbox{Integer} \mbox{(variant)} & $\X_{\R}^1$ & $\Z$ & $\left[-\frac{1}{2},\frac{1}{2}\right)$ & $\iota_-$ & Y&Y & \cite{Hurwitz} \\      \hline
     Folded Nearest Integer & $\X_{\R}^1$ & $\langle\Z,x\mapsto-x\rangle$ & $\left[0,\frac{1}{2}\right]$ & $\iota_+$ & Y&Y & \cite{MMY} \\
        \hline
     Nakada $\alpha$, $\alpha\in(0,1)$ & $\X_{\R}^1$ & $\Z$ & $[\alpha-1,\alpha]$ & $\iota_+$ & N&Y & cf.~\cite{AS,NakadaAlpha} \\ 
     \hline
     Even & $\X_{\R}^1$ & $2\Z$ & $[-1,1)$ & $\iota_-$ & Y & N & cf.~\cite{BL,KU2007} \\
      \hline
     Rosen for $q\in \mathbb{N}_{\ge 3}$ & $\X_{\R}^1$ & $\lambda \Z, \ \lambda=2\cos\frac{\pi}{q}$ & $\left[-\frac{\lambda}{2},\frac{\lambda}{2}\right)$ & $\iota_-$ & Y&Y  & \cite{MS}, cf.~\cite{BKS} \\
      \hline
     $\alpha$-Rosen for $q\in \mathbb{N}_{\ge 3}$& $\X_{\R}^1$ & $\lambda \Z, \ \lambda=2\cos\frac{\pi}{q}$ & $\left[\lambda(\alpha-1),\lambda\alpha\right)$, $\alpha\in[1/2,1/\lambda)$ & $\iota_-$ & Y&Y  & new, cf.~\cite{DKS} \\
     \hline
     \hline
     Hurwitz& $\X_{\R}^2$ & $\Z^2$ & $\left[-\frac{1}{2},\frac{1}{2}\right)^2$ & $\iota_c$ & N&Y & \cite{cijsouw,Hensley} \\
     \hline
     Folded \mbox{Hurwitz}  & $\X_{\R}^2$ & $\langle\Z^2,(x,y)\mapsto(-x,-y)\rangle$ & $\begin{aligned}&\left[-\tfrac{1}{2},\tfrac{1}{2}\right)\\ &\quad \times \left[-\tfrac{1}{2},0\right]\end{aligned}$ & $\iota_c$ & Y&Y & cf.~\cite{Pollicott}\\
          \hline
     Hurwitz Hexagonal  & $\X_{\R}^2$ & $\Z[\rho]$, with $\rho=\frac{1+\sqrt{-3}}{2}$ & Dirichlet region & $\iota_c$ & N&Y & \cite{MR1554754} \\
     \hline
     J.~Hurwitz or Tanaka& $\X_{\R}^2$ & $\{(a,b)\in\Z^2: a+b\text{ even}\}$ & Dirichlet region & $\iota_c$ & Y&N & \cite{cijsouw,MR800085} \\
          \hline
     Shallit  & $\X_{\R}^2$ & $\Z^2$ & See Rmk.~\ref{rmk:Shallit} & $\iota_c$ & N&N & \cite{cijsouw} \\
          \hline
     SKT  &$\X_{\R}^2$ & $\Z[\rho]$, with $\rho=\frac{1+\sqrt{-3}}{2}$ & $\begin{aligned}&\left[0,1\right)\rho  \\&\quad \times \left[0,1\right)\overline{\rho}\end{aligned}$ & $\iota_c$ & N&N & \cite{MR0429773} \\
     \hline
          Bianchi, $d=1,2,3,7,11$  & $\X_{\R}^2$ & $\mathcal{O}_d$, ring of integers & Dirichlet region & $\iota_c$ & N&Y & \cite{Dani,Hines} \\
     \hline
     \hline
          3d  & $\X_{\R}^3$ & $\Z^3$ & $\left[-\tfrac{1}{2},\tfrac{1}{2}\right)^3$ & $\iota_+$ & N&Y & new\\
     \hline
     \hline
     Quaternionic & $\X_{\R}^4$ & $\Z^4$ & $\left[-\tfrac{1}{2},\tfrac{1}{2}\right)^4$ & $\iota_c$ & N&N & \cite{Hamilton1,Hamilton2}\\
     \hline
     Hurwitz Quaternionic & $\X_{\R}^4$ & Hurwitz \mbox{integers} & Dirichlet region & $\iota_c$ & N&Y & \cite{Mennen}\\
     \hline
     \hline
     Octonionic & $\X_{\R}^8$ & Cayley \mbox{integers} & Dirichlet region & $\iota_c$ & N&Y & new\\
     \hline
     \hline
     Heisenberg & $\X_{\C}^1$ & $\Z^3$ & $\left[-\frac{1}{2},\frac{1}{2}\right)^3$ & $\iota_-$ & N&Y & \cite{LV}\\
     \hline
     Folded \mbox{Heisenberg} & $\X_{\C}^1$ & $\langle\Z^3,(z,t)\mapsto(\ii z,t)\rangle$ & $\begin{aligned}&\left[-\tfrac{1}{2},0\right]^2\\ &\quad \times \left[-\tfrac{1}{2},\tfrac{1}{2}\right)\end{aligned}$ & $\iota_-$ & Y&Y & new\\
     \hline     Heisenberg Hexagonal & $\X_{\C}^1$ & $\Z[\rho]\times \sqrt{3}\Z$ & See Ex. \ref{ex:HeisHex}  & $\iota_-$ & N&Y & new \\
     \hline
     \hline
          Heisenberg Quaternionic & $\X_{\mathcal{H}}^1$ & $(\Z^4\cup(\Z+1/2)^4)\times \Z^3$& Dirichlet region & $\iota_-$ & N&N & new\\
     \hline
\end{tabular}
\end{table}

It should be noted that all cases under consideration are discrete. 

In some cases where the fundamental domain is too complicated to write succinctly, we have labeled it with the Dirichlet region. In this case, we mean the set of points that are closer to $0$ than to any translate of $0$ under $\Zee$, with some choice of boundary. 

\begin{remark}\label{rmk:Shallit}
Note as well that the fundamental domain $K$ for the Shallit complex CF algorithm is a rectangle with corners at $.5-.5\ii$, $1$, $\ii$, and $-.5+.5\ii$ \cite{cijsouw}.
\end{remark}

The complex continued fractions, quaternionic continued fractions, and octonionic continued fractions are embedded in higher-dimensional real spaces in the standard way, $\C\cong \R^2$, $\mathcal{H}\cong \R^4$, and $\mathbb{O}\cong \R^8$. The inversion $\iota_c$ listed in all these cases is equivalent to $z\mapsto 1/z$ on $\C$, $\mathcal{H}$, or $\mathbb{O}$. 
One reason for identifying these spaces is that the existence of maximal orders, the Gaussian and Eisenstein integers in $\mathbb{C}$, the Hurwitz integers in $\mathcal{H}$, and the Cayley integers in $\mathbb{O}$, give rise to lattices on $\mathbb{R}^2$, $\mathbb{R}^4$, and $\mathbb{R}^8$ that in turn generate \emph{proper} fundamental domains $K$.
The Hurwitz integers in $\mathcal{H}$ are given by
\[\label{eq:Hurwitzintegers}\{a+b\ii+c\mathbbm{j}+d\mathbbm{k}:a,b,c,d\in\Z \text{ or }a,b,c,d\in \Z+1/2\}
\]
The Cayley integers in $\mathbb{O}$ are defined in Chapter 9 of \cite{CSbook} (where they are referred to by the less common name of octavian integers), with properness of the corresponding Dirichlet region following from Lemma 6 of that chapter. 

We should emphasize that Table \ref{tab:examples} does not cover all well-studied CF algorithms. For example, odd CFs \cite{BM}, CFs related to triangle groups \cite{CS}, CFs related to the Jacobi-Perron algorithm or other subtraction algorithms \cite{SchweigerBook}, regular chains \cite{Schmidt}, and general $(a,b)$-continued fractions \cite{KU2012} do not fit into our framework. The $N$-continued fractions \cite{DKvdW} and $u$-backwards continued fraction \cite{GH} use an $\iota$ which is not an inversion by our definition; however, our proofs could be modified to compensate. Regardless, they would still not be proper.

\begin{remark}
We are not the first to encounter problems with the incompleteness of the Hurwitz CF algorithm. Pollicott \cite{Pollicott} studied a similar folded continued fraction, albeit using conjugation in place of negation. Nakada \cite{Nakada} studied the full Hurwitz CF, but took as his hyperbolic space the disjoint union of two different spaces and let negation additionally act by swapping between the two.
\end{remark}

\subsection{Discreteness and Properness}

The difficulty of pushing into higher dimensions (either by taking $k\neq \R$ or $n\geq 2$) is in finding an appropriate lattice $\Zee$ and fundamental domain $K$ such that the resulting continued fraction is both discrete and proper.

The following proposition gives a useful framework for which to prove discreteness:
\begin{prop}\label{prop:easydiscrete}
Fix an Iwasawa inversion space $\X=\X^n_k$, an inversion $\iota$ that is either $\iota_+$, $\iota_-$, or $\iota_c$, and a discrete subring $R\subset k$ such that $2\in R$. Consider the subgroup $\Zee\subset \Isom(X)$ consisting of left-translations by points $(z,t)\in \X$ such that $z\in R^n$ and  $\Norm{z}^2+t\in R$. Then $\Mod=\langle \Zee, \iota\rangle\subset \Isom(\Hyp)$ is discrete.
\end{prop}

\begin{example}
\label{ex:Heisenberg}
For example, in the case of the first Heisenberg group $\X^1_\C$, we might chose $R=\mathbb{Z}[\ii]$ so $z\in \mathbb{Z}[\ii]$ and $t\in \ii\mathbb{Z}$. 
\end{example}

\begin{proof}
We can embed $\Mod$ as a subgroup of $GL(n+2,k)$ by mapping $\iota$ to a matrix $J_\iota$ of the form \eqref{eq:inversionmatrix}, and left-translation by $(z,t)$ to the matrix $A_{(z,t)}$, where
\[\label{Eq:Aztmatrix}
A_{(z,t)}=\left[ \begin{array}{ccc} 1 & 0_n & 0 \\ \sqrt{2}z & \operatorname{id}_n & 0_n \\ \Norm{z}^2+t & \sqrt{2}\overline{z} & 1 \end{array}\right].
\]
It is now easy to check that $\Zee$ is a group.

Unless $\sqrt{2}\in R$, the matrices $A_{(z,t)}$ will not be matrices over $R$ itself. However, consider the discrete set $S$ of $(n+2)\times(n+2)$ matrices $(a_{i,j})_{i,j=1}^{n+2}$ such that $a_{i,j}\in \sqrt{2}R$ if $i$ or $j$ (but not both!) is equal to $1$ or $n+2$, and otherwise $a_{i,j}\in R$. It is easy to check that $S$ is closed under multiplication. Moreover, the generators $J_\iota$ and $A_{(z,t)}$ of $\Mod$ belong to $S$, so that $\Mod\subset S$, so $\Mod$ must be discrete.
\end{proof}

For the rest of this section, we will assume that all the hypotheses of Proposition \ref{prop:easydiscrete} are satisfied, so that the only remaining difficulty is proving properness. 

\begin{example}\label{ex:3d}
Let us consider higher-dimensional generalizations of the nearest-integers CFs. Let $k=\R$, $\X=\X_\R^n=\R^n$, for some $n\ge 1$, and $\iota=\iota_+$. The space $\R^n$ admits the standard lattice $\Zee=\Z^n$ with fundamental domain $K=[-1/2,1/2)^n$.

When $n=1$, we get the usual nearest-integer CFs. When $n=2$, we get a variant of the Hurwitz complex CFs ($\iota_+$ acts like $z\mapsto 1/\overline{z}$). When $n=3$, we get a 3d CF which we do not believe has been studied before. However, when $n\ge4$, the corresponding $K$ is no longer proper.
\end{example}

Examples \ref{ex:Heisenberg} and \ref{ex:3d} fit into the framework of Proposition \ref{prop:easydiscrete} very easily. However, in general, $t$ may not belong to the ring $R$, but does belong to the additive subgroup $R'$ of $\Im(R)$ defined by
\(
R'=\{t\in \Im(R): \Norm{z}^2+t\in R, \exists z\in R^n\}\subset  \Im(R).
\)
One shows that, as a set, we have $\Zee=R^n\times R'$. 

Let $K_1$ be the Dirichlet domain around $0$ for $R$ and let $K_2$ be the Dirichlet domain around $0$ for $R'$ with respect to the Euclidean metrics on $k^n$ and $\Im(k)$. Then a fundamental domain for $\Zee$ in $\X$ is given by $K=K_1^n\times K_2$. In particular, the radius of $K$ is 
\(
\rad(K)=\sqrt[4]{n^2\rad(K_1)^4+\rad(K_2)^2}.
\)
Thus, to obtain a proper system, we require $n^2\rad(K_1)^4+\rad(K_2)^2<1$.

\begin{example}\label{ex:HeisHex}
Suppose $k=\mathbb{C}$ and $R=\Z[\ii]$. Then we have  $R'=\ii\Z$, $K_1=[-1/2,1/2)^2$, $K_2=[-1/2,1/2)\ii$. In this case $\rad(K_1)=2^{-1/2}$ and $\rad(K_2)=2^{-1}$. When $n=1$, this implies that $K$ is proper, and results in the Heisenberg continued fractions in Table \ref{tab:examples} above. However, $\rad(K)<1$ only for $n=1$ and so this cannot be directly generalized to higher Heisenberg groups. 

It is tempting to get around this by replacing $R$ with $\Z[e^{2\pi \ii/3}]$, the Eisenstein integers, as then $K_1$ is a hexagon with radius $3^{-1/2}$. However, this gives $R'=\sqrt{3}\ii \Z$, so that $K_2=[-\sqrt{3}/2,\sqrt{3}/2)\ii$, and again $\rad(K)<1$ only for $n=1$.

We would, more generally, be interested in CFs on the Heisenberg group with coordinates related to the ring of integers of imaginary quadratic fields. However, if we use $R=\mathcal{O}_d$ for $d=2,7,11$, then the resulting fundamental domain $K_1\times K_2$ is not proper even when $n=1$.
\end{example}

\begin{example}\label{ex:hurwitzquat}
Let $k=\mathcal{H}$ be the quaternions, $n=1$, and $R$ the Hurwitz integers \eqref{eq:Hurwitzintegers}, so that $R'=\Z[\ii,\mathbbm{j},\mathbbm{k}]$. Then $\rad(K_1)=2^{-1/2}$ (see \cite{Mennen}) and $K_2=[-1/2,1/2)^3$ so $\rad(K_2)=\sqrt{3}/2$. In particular, if we look at $X^1_\mathcal{H}$, we have $\rad(K)=1$, narrowly missing the properness criterion. Other nearly-proper CF algorithms such as the J.~Hurwitz complex CFs are known to be convergent and ergodic, so we hope to be able to extend our results to this case.
\end{example}

\subsection{Completeness and Incompleteness}\label{sec:appendix}

We now demonstrate how one can identify complete CFs, or identify symmetries of incomplete CFs.

\begin{prop}\label{prop:firstcomplete}
CF algorithms  associated to $\X^1_\R$, $\Zee=\Z$, and $\iota_+(x)=1/x$ (e.g., regular or $\alpha$-CFs) are incomplete with two central symmetries. CF algorithms associated to  $\X^1_\R$, $\Zee=\Z$, and $\iota_-(x)=-1/x$ (e.g., backwards) are complete.
\begin{proof}
Let $\Mod_+$ and $\Mod_-$ be the modular groups associated to $\iota_+$ and $\iota_-$, respectively. We take advantage of the fact that one can embed $\Mod_-$ into $SL(2,\Z)$, while $\Mod_+$ naturally embeds into the larger $GL(2,\Z)$.

That is, we may identify elements of $\Zee$ and the inversions $\iota_\pm$ with matrices in $GL(2,\mathbb{Z})$, acting by the usual linear fraction transformations on $\R$, with
\(
\Zee=\left\{ A_n=\left(\begin{array}{cc} 1 & n \\ 0 & 1 \end{array}\right): n\in \mathbb{Z}\right\} \qquad \iota_\pm = \left( \begin{array}{cc} 0 & \pm 1 \\ 1 & 0 \end{array}\right).
\)
(Note that in the standard convention, translations  act by upper-triangular matrices, cf.~\eqref{Eq:Aztmatrix}.) 
To test for completeness, note that matrices in $\Stab_{\Mod_\pm}(\infty)$ have the form
\(
\left( \begin{array}{cc} a & b\\ 0 & d\end{array}\right).
\)
Since $a,d\in \Z$ and $\norm{ad}=1$, $a,d$ must be units, so we can decompose the matrix as 
\(
\left( \begin{array}{cc} a & b\\ 0 & d\end{array}\right)=\left( \begin{array}{cc} 1 & b(d^{-1}) \\ 0 & 1\end{array}\right) \left( \begin{array}{cc} a & 0\\ 0 & d\end{array}\right),
\)
a product of an element of $\Zee$ and a diagonal matrix. So the only things that can potentially cause incompleteness are diagonal matrices in $\Mod$. Since the only diagonal matrices in $SL(2,\Z)$ are $\pm I$, which act by the identity, we can conclude $\Mod_-=\Stab_{\Mod_-}(\infty)$.

For $GL(2,\Z)$, the only potential additional symmetry is given by $x\mapsto -x$, corresponding to a diagonal matrix with $a=-d$. Indeed, this is contained in $\Mod_+$, represented by the word $\iota A_1 \iota A_{-1} \iota A_{1}$. In particular, CFs associated with $\iota_+$ are incomplete with $2$ central symmetries.
\end{proof}
\end{prop}

A proof similar to the above also implies that the Rosen CFs are complete.

\begin{prop}\label{prop:secondcomplete}
Let $k$ be the complex, quaternionic, or octonionic division algebra, with $\Zee$ given by translation by Gaussian  or Eisenstein integers, quaternionic or Hurwitz integers, or Cayley integers respectively. Any $k$-CFs with associated with $\Zee$ and an inversion of either $z\mapsto 1/z$ or $z\mapsto -1/z$ is incomplete with at least two central symmetries.
\begin{proof}
One argues along the same lines as the proof of Proposition \ref{prop:firstcomplete}, embedding $\Mod$ into $GL(2,\mathcal{O}_k)$, where $\mathcal{O}_k$ is the corresponding ring of integers. If $\iota(z)=1/z$ then $\iota A_1 \iota A_{-1} \iota A_{1}$ is again the central symmetry $z\mapsto -z$. If $\iota(z)=-1/z$, then the central symmetry $z\mapsto -z$ can be represented by the word $\iota A_{i} \iota A_{-i} \iota A_i$. In the Hurwitz complex CF case, no other central symmetries can be obtained because the matrices of $GL(2,\mathcal{O}_k)$ obtained by the embedding have determinant $\pm 1$, and hence the only diagonal matrices have $a=\overline{d}$ or $a=-\overline{d}$.
\end{proof}
\end{prop}

\begin{prop}
The J.~Hurwitz complex CF algorithm is complete.
\begin{proof}
As in Proposition \ref{prop:secondcomplete}, we embed $\Mod$ into $GL(2,\Z[\ii])$, with $\iota=\iota_+$ and $\Zee=\{A_n:n\in(1+\ii)\Z[\ii]\}$. However, by taking $\Mod$ modulo $4$ and performing an exhaustive computational search, one can confirm that the central symmetry $z\mapsto -z$ never appears.
\end{proof}
\end{prop}

\begin{prop}
Standard Heisenberg continued fractions are incomplete with four central symmetries.
\begin{proof}
Embed $\Mod$ into $GL(3,\Z[\ii])$ using \eqref{Eq:Aztmatrix}. Diagonal matrices then correspond to the rotations $(z,t)\mapsto (\ii^k z,t)$. All four of these are, in fact, realized, since one has
\(
\iota A_{(0,1)}\iota A_{(0,1)} \iota A_{(0,1)} = \left( \begin{array}{ccc} -\ii & 0 & 0\\ 0 & 1 & 0 \\ 0 & 0 & -\ii\end{array}\right),
\)
corresponding to the rotation $(z,t)\mapsto(\ii z,t)$.
\end{proof}
\end{prop}

\section{Convergence}\label{sec:convergence}

Convergence in the specific case of proper and discrete Iwasawa continued fractions with $k=\C$, $n=1$, and $\Zee$ left-translations by the integer Heisenberg group was given in \cite{LV}, Lemma 3.19 through Theorem 3.21. We now extend this to the following result:

\begin{thm}
\label{thm:convergence}
Fix a proper and discrete Iwasawa continued fraction algorithm, and let $x\in K$. If $x$ has infinitely many CF digits, then the convergents $M_i(0)$ converge to $x$; otherwise, if $x$ has exactly $i$ CF digits, then $M_i(0)=x$.
\end{thm}

As the proof is nearly identical to that in \cite{LV} with some notational changes, we only highlight the general method and the new aspects of the proof. In \cite{LV}, convergence is proven by extending a regular CF formula for the distance between a point and its convergents, which  reads as follows:
\(
d(x,M_i(0))= \dfrac{\prod_{j=0}^i \norm{T^j x}}{\Norm{q_i}^{1/2}}
\)
where $q_i$ is the denominator of $M_i(0)$ (see below for a more precise definition). The proof of this formula extends unchanged from the Heisenberg case, and since $\norm{T^j x}\le \rad K<1$ for proper Iwasawa CFs, convergence is immediate provided $\Norm{q_i}$ is bounded away from $0$. It is this last point where new techniques are required.  In Lemma 3.20 of \cite{LV}, the discreteness of the Gaussian integers was used to prove that $q_i\neq 0$, and thus, since $q_i\in\mathbb{Z}[\ii]$, we must have that $\Norm{q_i}\ge 1$. However, in the general Iwasawa CF case, the rings generated by the coefficients of $\Mod$ (in a given matrix representation) need not be discrete, so a new technique is needed.

We proceed by first fixing a proper and discrete Iwasawa continued fraction algorithm. Note that we will not use properness explicitly, but it is necessary for the remainder of the proof in \cite{LV}.

Recall from \S \ref{sec:hyp} that $\Hyp$ is the set $\{h=(z,w)\in k^n\times k \st \Re(w)>0\}$ with boundary $\partial \Hyp=\X$. The coordinate $\Re(w)$ is the \emph{horoheight (at infinity)} $\height_\infty(h)$. Restricting horoheight from below produces a \emph{horoball at $\infty$}, and applying a mapping $M\in \Mod$ produces a horoball at the point $M(\infty)$. These can be defined directly using the horoheight  $\height_{M(\infty)}(h):=\height_{\infty}(M^{-1}(h))$. It follows from the characterization of horoballs as limits of metric balls that horoballs are geodesically convex. We denote the horoball of height $C$ based at a point $M(\infty)$ by $\mathcal B_{M(\infty)}(C)=\{h\in \Hyp\st \height_{M(\infty)}(h)\geq C\}$.

The following generalizes the disjointness result for Ford circles:
\begin{thm}
\label{thm:thick-thin}
There exists $C_0>0$ such that for every $C\geq C_0$ and $M_1, M_2\in \Mod$ satisfying $M_1(\infty)\neq M_2(\infty)$, the horoballs $\mathcal B_{M_1(\infty)}(C)$ and $\mathcal B_{M_2(\infty)}(C)$ are disjoint.
\begin{proof}[Sketch of Proof]
The result follows from the Margulis Lemma by way of the Thick-Thin Decomposition (see, e.g., \S5.10 of Thurston's notes \cite{ThurstonNotes}) of the quotient orbifold $\Mod\backslash \Hyp$, which has a cusp corresponding to the point $\infty$. To see that it has this cusp, note that the translation length for elements of $\Zee\subset \Hyp$ goes to zero at large horoheight (note that one can compare actions at different horoheights by conjugating by the dilation $\delta_r$), so that a horoball of sufficiently large horoheight must be contained in the thin part of $\Mod \backslash \Hyp$.
\end{proof}
\end{thm}

We can conclude, in particular, that horoballs based at points other than $\infty$ are quantitatively bounded with respect to horoheight from $\infty$.
\begin{cor}
\label{cor:smallballs}
Let $\mathcal B=\mathcal B_\infty(h_1)$ be a horoball of height $h_1$ based at $\infty$. Then for every $M\in \Mod$ satisfying $M(\infty) \neq \infty$, one has \(\height_\infty(M(\mathcal B)):=\sup \{\height_\infty(h) \st h\in M(\mathcal B)\}\le C_0^2/h_1=:h_2.\)
\end{cor}
\begin{proof}We first show that for each $M\in \Mod$ there exists a $C_M>0$ such that $\height_\infty(M (\mathcal B_\infty(h)))=C_Mh^{-1}$ for each $h>0$.  To verify this, we use the fact that $\Mod=\langle \Zee, \iota\rangle$ to expand $M=\iota  a_n \cdots a_1 \iota$ for $a_i\in \Zee$, noting that initial and final translations don't affect horoheight. On the other hand, each inversion acts, by Lemmas 3.6 and 3.8 of \cite{1510.06033}, via:
\(\height_\infty(\iota (\mathcal B_\infty(h)))=1/h, \hspace{.75in} \height_\infty(\iota (\mathcal B_x(h))) = h \norm{x}^{-2}.\)
Thus, as long as, for each $i$, $x_i:=(a_i \iota  \cdots  a_1  \iota )(\infty) \neq 0$, we have 
\(\height_\infty(M(\mathcal B_\infty(h)) = h_1^{-1} \prod_{i=1}^{n} \norm{x_i}^{-2}.\)
If at some point $x_i=0$, then we must have $(\iota a_i\iota  \cdots a_1 \iota)(\mathcal B_\infty(h))=\mathcal B_\infty(h)$, so that digits $a_1, \ldots, a_i$ may be removed without altering the effect of $M$ on $\mathcal B_\infty(h)$. With the reduction implemented, the product $C_M:=\prod_{i=1}^{n} \norm{x_i}^{-2}$ is well-defined and has the desired property.

To complete the argument, note that from Theorem \ref{thm:thick-thin} we have that $h^{-1}C_M<h$ for $h=C_0$, so $C_M<C_0^2$ and $\height_\infty(M(\mathcal B))<h_2$, as desired.
\end{proof}

Recall that we have an embedding $\phi:\X\to k^{n+2}$ given by $\phi(z,t)=(1,\sqrt{2}z,\Norm{z}^2+t)$; with a corresponding embedding of $\Mod$ into $U(J)\subset GL(n+2,k)$ acting on these vectors. Isometries of $\X$ then embed as lower block triangular mappings of the form
\(
\begin{bmatrix}
    a & 0_n & 0\\
    b & A & 0_n\\
    c & b^\dagger & \overline a
\end{bmatrix},
\)
where $\norm{a}=1$ and $A$ is a unitary transformation. The matrix associated to the inversion is given by Lemma \ref{lemma:inversionform}.

Now, given a point $x\in K$ with at least $m$ continued fraction  digits (note that \cite{LV} uses the variable $n$ instead), let $q_m$ be the denominator of $M_m(0)$; that is, the first coordinate of the vector $M_m\phi(0)$. Thus in the matrix representation of $M_m$, the top-left entry is $q_m$ and the top-right entry, in norm, is $\Norm{q_{m-1}}$, matching the matrix representation in Lemma 3.16 of \cite{LV}.

\begin{lemma}
\label{lemma:discreteness}
Under the assumptions of Theorem \ref{thm:convergence}, there exists $C>0$ such that $q_m\neq 0$ implies $\Norm{q_m}>C$.
\begin{proof}
By Theorem \ref{thm:thick-thin}, there exists a horoball $\mathcal B$ based at $\infty$ of some horoheight $C_1$ such that the $\Mod$-orbit of $\mathcal B$ consists of disjoint horoballs. Moreover, the proof of Lemma 3.9 of \cite{1510.06033} (again, readily extended to the current setting) gives a constant $s_0$ such that if $q_m \neq 0$ then \(\height_\infty(M_m(\mathcal B)):=\sup\{\height_\infty(h)\st h\in M_m(\mathcal B)\}\geq s_0\Norm{q_m}^{-1}.\)
The disjointness requirement forces $\height_\infty(M_m(\mathcal B))<C_1$, so $\Norm{q_m}>s_0/C_1=:C$. 
\end{proof}
\end{lemma}

From here, it remains to show that $q_m\neq 0$. This is just the content of Lemma 3.20 of \cite{LV} and we can extend the argument to the general case by citing Lemma \ref{lemma:discreteness} above in place of the fact that non-zero Gaussian integers have norm at least 1.

\section{Markable Geodesics}\label{sec:markable}
We now study the way a geodesic $\gamma$ interacts with the modular group $\Mod$ related to a proper, discrete, and complete Iwasawa continued fraction algorithm, with the goal of proving the Markable Geodesic Theorem \ref{thm:markable} below.  We will track the passage of a geodesic through $\Mod\backslash\Hyp$ by detecting intersections with the unit sphere
\(\Sph=\{ h \in \Hyp \st \norm{h}=1\}\)
and its images under elements of $\Mod$. We will obtain an analog of geodesic coding for certain \emph{markable} geodesics, and then show that markability is a generic condition. Note that $\partial \Sph$ is the unit sphere in $\X$, and that $\iota(\Sph)=\Sph$.

\begin{thm}[Markable Geodesic Theorem]\label{thm:markable}
Fix a complete, proper, and discrete Iwasawa CF algorithm on an Iwasawa inversion space $\X$, with the associated hyperbolic space $\Hyp$, modular group $\Mod$, and fundamental domain $K\subset\X$ for the lattice $\Zee=\Stab_\Mod(\infty)$. 

There exists a codimension-one set $\CW \subset T^1 \Hyp$ and a \emph{marking} that assigns to every markable geodesic satisfying $\gamma(0)\in \CW$
\begin{itemize}
    \item digits $a_i\in \Zee$ and mappings $M_i\in \Mod$, for each $i\in \Z$, 
    \item increasing  indices $i_j\in \Z$ and times $t_j$, for each $j\in \Z$, with $i_0=0, t_0=0$
\end{itemize}
collectively called the marking of the geodesic $\gamma$ such that:
\begin{enumerate}
    \item (Full Coverage) The segments $[t_{j-1}, t_j]$ have length uniformly bounded below and hence cover all of $\R$,
    \item (Relation to Shift Map) For each $i\geq 1$,  $a_i$ is the $i^{th}$ CF digit of $\gamma_+$, and  $M_i$ is the branch of $T^{-i}$ associated to the shift map $T$ at $\gamma_+$,
    \item (Cusp Detection)
    If, for $t\in [t_{j-1}, t_j]$, the horoheight of $\gamma(t)$ from $M\infty$ satisfies $\height_{M\infty} \gamma(t)>h_0$, and  if $M^{-1} \gamma_+\in K$ for some $M\in \Mod$, then $M=M_{i_j}$,
    \item (Intersection Detection) Let $M\in \Mod$ and  $t\in \R$. Then one has $\gamma(t)\in M\CW$ if and only if for some $j$ one has $t=t_j$ and $M=M_{i_j}$,
    \item (Shifted Gauss Equivariance) Let $k\in \Z$. The marking $\{a_i',M'_i,i'_j, t'_j\}$ associated to the markable geodesic $\gamma'(t):=M_{i_k}^{-1}\gamma(t+t_k)$ satisfies: $t'_j=t_{j+k}-t_k$, $i'_j=i_{j+k}-i_k$, $a'_i=a_{i+i_k}$, and $M'_i=M_{i_k}^{-1}M_{i+i_k}$.
\end{enumerate}
\end{thm}

To begin with, in $\Hyp_\R^2$, it is apparent from the geometry that any geodesic can only intersect $\Sph$ transversely at a single point; however, in other hyperbolic spaces, even a generic geodesic may intersect $\Sph$ at more than one point; indeed when $k\neq \R$, $\Hyp$ does not admit \emph{any} geodesically convex codimension-1 hypersurfaces. However, a generic geodesic intersects $\Sph$ in finitely many points, so we may speak of the \emph{last} intersection with $\Sph$:

\begin{lemma}\label{lemma:finiteintersection}
Let $\gamma$ be a geodesic in $\Hyp$ not contained in $\Sph$. Then the set of intersections $\gamma\cap \Sph$ is finite. Furthermore, if there are times $t_1, t_2$ such that $\norm{\gamma(t_1)}>1$ and $\norm{\gamma(t_2)}<1$, then $\gamma$ does intersect $\Sph$.
\begin{proof}
The existence of the intersection follows from the definition of $\Sph$ by $\norm{\cdot}=1$.

Finiteness follows by an algebraic argument. Because $\Isom(\Hyp)$ acts transitively on geodesics, we may write $\gamma=g(\gamma_2)$, where $g\in G$ and $\gamma_2$ is the geodesic joining $0$ and $\infty$. Because $g$ and acts by projective transformations on $\Hyp$, the condition $\norm{g(\gamma_2(t))}=1$ induces an algebraic condition on $t$. Thus, if the condition were to be satisfied for infinitely many $t$, it must be satisfied for all $t$, so that $\gamma\subset \Sph$, a contradiction.
\end{proof}
\end{lemma}

We now establish the necessary results for the proof of the Markable Geodesic Theorem.

\subsection{Decomposing an Arbitrary Geodesic}\label{subsec:geodesic}
In the first stage of the proof, we will break up a geodesic $\gamma$ into segments punctuated by intersections with expected images of the sphere $\Sph$, in a way that gives us control of the intermediate horoheights. For a more formal statement, see Lemma \ref{lemma:upshot} below.

We start by restricting our attention to geodesics that intersect near the top of $\Sph$. Fix $\epsilon>0$ such that $\epsilon+1< \rad(K)^{-1}$ (this choice comes into play in Lemma \ref{lemma:exile}). We then have:
\begin{lemma}
\label{lemma:horocap}
Suppose $\gamma$ is a geodesic ray with $\norm{\gamma(0)}\ge 1+\epsilon$ and $\gamma_+\in K$. Then the horoheight of any intersection of $\gamma$ with $\Sph$ satisfies $\height_{\infty}(\gamma(t))\ge h_2$ for some $h_2\in(0,1)$ depending only on $\epsilon$.
\begin{proof}
The existence of the intersection follows from Lemma \ref{lemma:finiteintersection}.

To obtain the lower bound on the horoheight of each intersection, note that $\gamma$ is uniformly transverse to boundary $\X$ (note that we are not working in a conformal model, so $\gamma$ is not necessarily perpendicular to $\X$), as this is true for the vertical geodesic joining $0$ and $\infty$ and the endpoints of $\gamma$ are contained in the compact set $\overline{K}\times (\overline \Hyp \setminus B(0,1+\epsilon))$. Thus, there is a minimal horoheight $h_2$ (that we may assume is in $(0,1)$) that $\gamma$ must reach as it moves away from $\gamma_-$ and $\gamma_+$ before an intersection can occur. The same bound must hold for the intermediate segment by the convexity of horoballs.
\end{proof}
\end{lemma}

We denote the subset of $\Sph$ having horoheight at least $h_2$ as $\W$, and refer to both $\W$ and its images under $\Mod$ as ``walls''.

We next fix a geodesic ray $\gamma$ originating in $\W$ and terminating in $K$ and let $M_i\in\Mod$ be the mappings associated to the CF expansion of $\gamma_+$. We now look for intersections of $\gamma$ with walls $M_i( \W)$ by iterating the shift map on $\gamma$ and identifying intersections of $M_i^{-1}(\gamma)$ with $\W$. This happens within finitely many iterations, with control over the intermediate digits:

\begin{lemma}
\label{lemma:exile}
There is a finite collection $\Mod_0\subset \Mod$ such that the following holds. Suppose $\gamma$ is a geodesic with $\gamma(0)\in \W$ satisfying $\gamma_+\in K\setminus \Mod \infty$. Then there exists a time $0<\tee_1<\infty$ and a universally bounded ${\ai_1}>0$ such that $M_{\ai_1}^{-1}(\gamma(\tee_1))\in \W$ and $M_{\ai_1-1}\in \Mod_0$.
\end{lemma}

At this point, for notational convenience, we will often drop parentheses when elements of $\Mod$ act on points or sets of points. 

\begin{proof}
We note first that since $\gamma_+\not\in \Mod\infty$, then the continued fraction expansion of $\gamma_+$ does not terminate and so $M_i^{-1} \gamma_+$ is well-defined and in $K$ for all $i\in \mathbb{N}$.

If $\norm{M_1\gamma(0)}\geq 1+\epsilon$, the result is immediate by Lemma \ref{lemma:horocap}.

If not, we proceed iteratively on $i$, starting at $i=1$, supposing at every stage that $\norm{M_{i-1}^{-1}\gamma(0)} < 1+\epsilon$ until we find the minimum positive $\ai_1$ for ${i}$ for which \[\norm{M_i^{-1}\gamma(0)}\ge 1+\epsilon.\label{eq:ibound}\] Note that $M_i^{-1}=a_i^{-1}\iota M_{i-1}^{-1}$, $M_0=\id$, and moreover that $a_i^{-1}$ is an isometry of the metric $d$.

When $i=1$, we have by the above observation and our definition of inversions that
\[
d(M_1^{-1}\gamma_+,M_1^{-1}\gamma(0))&= d(\iota M_{0}^{-1}\gamma_+,\iota M_{0}^{-1} \gamma(0))= \frac{d( M_{0}^{-1}\gamma_+, M_{0}^{-1} \gamma(0))}{\norm{M_{0}^{-1}\gamma_+}\norm{M_0^{-1}\gamma(0)}}\\
&\ge \frac{d( M_{0}^{-1}\gamma_+, M_{0}^{-1} \gamma(0))}{\rad(K)(1+\epsilon)}=\frac{d(\gamma_+,  \gamma(0))}{\rad(K)(1+\epsilon)}.\label{eq:radKdenom1}
\]
In particular, since $d(\gamma_+,\gamma(0))\ge d(K,\W)$ this implies that
\[
\norm{M_1^{-1}\gamma(0)}&\ge d(M_1^{-1}\gamma_+,M_1^{-1}\gamma(0))-\norm{M_1^{-1} \gamma_+}\\
&\geq \frac{d(K,\W)}{\rad(K)(1+\epsilon)}-\rad(K)\label{eq:radKdenom2}
\]
This lower inequality could be substantially improved if more was known about $M_0^{-1}\gamma_+$. In particular, if $\norm{M_0^{-1}\gamma_+}\le r$ for
\(
r=\frac{d(K,\W)}{(1+\epsilon)(1+\epsilon+\rad(K))},
\)
then we could replace the $\rad(K)$ in the denominator of \eqref{eq:radKdenom1} and \eqref{eq:radKdenom2} with $r$ and obtain that $\norm{M_1^{-1}\gamma(0)}\ge 1+\epsilon$, so that $i=1$ itself is the minimum index for which \eqref{eq:ibound} holds.

Now we begin the iteration. At every stage we see that
\(
d(M_i^{-1}\gamma_+,M_i^{-1}\gamma(0)) &\ge \frac{d(M_{i-1}^{-1}\gamma_+,M_{i-1}^{-1}\gamma(0))}{\norm{M_{i-1}^{-1}\gamma_+}\norm{M_{i-1}^{-1}\gamma(0)}}\\ &\ge \frac{d(M_{i-2}^{-1}\gamma_+,M_{i-2}^{-1}\gamma(0))}{\norm{M_{i-1}^{-1}\gamma_+}\norm{M_{i-1}^{-1}\gamma(0)}\norm{M_{i-2}^{-1}\gamma_+}\norm{M_{i-2}^{-1}\gamma(0)}}\\
&\dots\\
&\ge \frac{d(\gamma_+,\gamma(0))}{\prod_{j=0}^{i-1}\norm{M_{j}^{-1}\gamma_+}\norm{M_{j}^{-1}\gamma(0)}},
\)
and thus 
\[
\norm{M_i^{-1}\gamma(0)}\ge \frac{d(K,\W)}{(\rad(K)(1+\epsilon))^i}-\rad(K),\label{eq:iterativestep}
\]
noting again that if $\norm{M_{i-1}^{-1}\gamma_+}\le r$, then one copy of $\rad(K)$ in the denominator of the last inequality can be replaced with $r$. Thus $i$ satisfies \eqref{eq:ibound}.

Regardless of whether $\norm{M_{i-1}^{-1}\gamma_+}\le r$ at any stage, since $\rad(K)(1+\epsilon)<1$ by the initial choice of $\epsilon$, within a bounded number of steps independent of our choice of $\gamma$, the expression on the right of \eqref{eq:iterativestep} exceeds $1+\epsilon$. Thus, there must be a uniform bound on $\ai_1$ such that $\norm{M_{\ai_1}^{-1}\gamma(0)}>1+\epsilon.$

Moreover, we see that if ever in our iterative process, $\norm{M_{i-1}^{-1}\gamma_+}\le r$, then this $i$ must be the desired value $\ai_1$. Thus for $i=\ai_1$, we must have that $\norm{M_j^{-1}\gamma_+}>r$, $0\le j <i-1$. However, recall that $a_{j+1}=[{\iota M_{j}\gamma_+}] $. In particular, this tells us that $a_{j+1}$ must belong to a finite set of values for $0\le j <i-1$, and since $M_{i-1}=\iota^{-1} a_1 \iota^{-1} a_2\dots \iota^{-1} a_{i-1}$, there are finitely many options for what it could be.
\end{proof}

\begin{cor}
\label{cor:h1}
There exists a universal $h_1>0$ such that under the assumptions of the preceding lemma we have $\height_{\infty}(M_{\ai_1}^{-1}\gamma(t))>h_1$ for all $0\leq t \leq \tee_1$.
\begin{proof}
We already know that $\height_\infty(M_{\ai_1}^{-1}\gamma(\tee_1))\geq h_2$ since this point is contained in $\W$.

Let us next consider the possible horoheights of $M_{\ai_1}^{-1}\gamma(0) = a_{\ai_1}^{-1} \iota M_{\ai_1-1}^{-1}\gamma(0)$. The point $\iota M_{{\ai_1}-1}^{-1}\gamma(0)$ lies  in the relatively compact set $\cup\{\iota M^{-1}\W \st M\in \Mod_0\}$, so for some $h_3$ we obtain $\height_\infty (\iota M_{{\ai_1}-1}^{-1}\gamma(0))>h_3$. Since translation along $\X$ does not affect horoheight, we likewise have $\height_\infty(M_{\ai_1}^{-1}\gamma(0))>h_3$.

The lemma now follows with $h_1=\min(h_2,h_3)$ by convexity of horoballs.  
\end{proof}
\end{cor}

We are now able to characterize $M_{\ai_1}$ as the (essentially) unique element of $\Mod$ that can detect large horoheights along the geodesic segment between $\gamma(0)$ and $\gamma(\tee_1)$. Let us define an exceptional set $E\subset K$ by
\[\label{eq:Edefn}
E= K\cap \bigcup_{a\in \Zee\setminus \{\id\}} aK.
\]
Since $K$ is a fundamental domain for $\Zee$, $E$ has measure zero.

\begin{cor}\label{cor:h0}
There is an $h_0>1$ such that the following holds under the conditions of Lemma \ref{lemma:exile}, and for all $0\leq t\leq \tee_1$. If $M^{-1}\gamma_+\in K\setminus E$, $M_{\ai_1}^{-1} \gamma_+\in K\setminus E$, and $\height_{M\infty}(\gamma(t))>h_0$, then $M=M_{\ai_1}$.
\begin{proof}
The geodesic segment $M_{\ai_1}^{-1}\gamma([0,\tee_1])$ is contained in the horoball $\mathcal B=\mathcal B_\infty(h_1)$, and by Corollary \ref{cor:smallballs} there is an $h_0$ such that the points  of $M\mathcal B$ have horoheight based at $\infty$ of at most $h_0$ when $M\infty \neq \infty$. In particular, this applies to the geodesic segment.

Thus, if $\height_{M\infty}(\gamma(t))>h_0$ for any $0\leq t\leq \tee_1$, then we conclude that $M\infty=M_{\ai_1}\infty$ and thus that $M^{-1}M_{\ai_1}\in \Stab_\Mod(\infty)=\Zee$, by completeness. Moreover, $\gamma_+\in M(K\setminus E)\cap M_{\ai_1}(K\setminus E)$ so that $M(K\setminus E)\cap M_{\ai_1}(K\setminus E) \neq \emptyset$. Thus $(M^{-1}M_{\ai_1} (K\setminus E))\cap (K\setminus E) \neq \emptyset$. By the definition of $E$, the only element of $\Zee$ that takes any part of $K\setminus E$ back to itself is the identity element. Thus $M=M_{\ai_1}$ as desired. 

We may assume without loss of generality that $h_0>1$.
\end{proof}
\end{cor}

Iterating the above results gives us a sequence of indices $\ai_j$ and times $\tee_j$ with the following properties:
\begin{lemma}\label{lemma:upshot} Let $h_0$ be the constant in Corollary \ref{cor:h0} and $\gamma$ a geodesic ray with $\gamma(0)\in \W$, $\gamma(t)\not\in \W$ for $t>0$, and $\gamma_+\in K\setminus \Mod (\{\infty\}\cup E)$. Then there is an increasing sequence $\ai_j$, $j\geq 0$, of indices starting with $\ai_0=0$ and an increasing sequence of times $\tee_j$, $j\geq 0$, starting with $\tee_0=0$ such that:
\begin{enumerate}
    \item For each $j\geq 0$: $\gamma(\tee_j)\in M_{\ai_j}\W$, while for $t>\tee_j$, $\gamma(t)\not\in M_{\ai_j}\W$,
    \item For each $j\geq 1$: If $\tee_{j-1}\leq t\leq \tee_j$ and a matrix $M\in \Mod$ satisfies both $M^{-1}\gamma_+\in K$ and $\height_{\infty}{M^{-1}\gamma(t)}>h_0$, then $M=M_{\ai_j}$.
\end{enumerate}
\begin{proof}Given $\gamma$ satisfying the assumptions of the lemma, the $j=0$ case of Conclusion (1) is trivial.

Moreover, we obtain $\ai_1$ and $\tee_1$ from Lemma \ref{lemma:exile}. There might be several choices of $\tee_1$ due to multiple intersections with $M_{\ai_1}\W$ (see Lemma \ref{lemma:finiteintersection}); however, we let $\tee_1$ be the last of these.  We then know that $M_{\ai_1}^{-1}\gamma(\tee_1)\in \W$, which is equivalent Conclusion (1) for $j=1$. We then obtain Conclusion (2) for $j=1$ from Corollary \ref{cor:h0}.

We now proceed inductively: once $\tee_j$ and $\ai_j$ are defined, we replace $\gamma$ with the geodesic segment $\gamma'(t')= M_{\ai_j}^{-1}\gamma(t'+\tee_j)$ restricted to $t'\in [0,\infty)$. We then obtain $\tee'_1$, $\ai'_1$, and $M_{\ai'_1}$ as before, and take $\tee_{j+1}=\tee_j+\tee'_1$ and $\ai_{j+1}=\ai_j+\ai'_1$. The desired properties follow from the fact that the shift map acts as a shift on the digits of $\gamma_+$, via the identity $M_{\ai_{j+1}}=M_{\ai_j}M'_{\ai'_1}$.

Finally, we note that since $h_0>1$, if $\height_{\infty}{M^{-1}\gamma(t)}>h_0$, then $t$ cannot be any of the $\tee_j$'s, so there is no ambiguity in Conclusion (2).
\end{proof}
\end{lemma}

\subsection{Decomposing a Markable Geodesic}\label{subsec:markable}

Lemma \ref{lemma:upshot} tells us how geodesic rays leaving the wall $\W$ towards $K$ return to other walls $M\W$, for various $M\in \Mod$. In particular, if a point on our ray has large horoheight with respect to $M\infty$, then the ray should cross the wall $M\W$. We now use this to define a set $\CW\subset T^1 \Hyp$ lying over $\W$, where this ``if'' condition becomes ``if and only if.'' We will then call a geodesic \emph{markable} if it intersects $\Mod$-translates of $\CW$ infinitely often in the past and future, and show in the Markable Geodesic Theorem (Theorem \ref{thm:markable}) that the behavior of a markable geodesic's cusp excursions is directly related to the continued fraction expansion of the forward endpoint. We will see in Corollary \ref{cor:markablegeneric} that markable geodesics are generic.

\begin{defi}
\label{defi:CW}
Using the constant $h_0>1$ provided by Lemma \ref{lemma:upshot}, we define $\CW \subset T^1\Hyp$ as follows: a vector based at a point in $\W$ is in the set $\CW$ if and only if the corresponding geodesic \emph{line} $\gamma$ satisfies:
\begin{enumerate}
\item $\gamma(0)\in \W$, while for $t>0$, $\gamma(t)\notin \W$,
\item $\gamma_+\in K\setminus \Mod(\{\infty\}\cup E)$, where $E$ is the exceptional set \eqref{eq:Edefn},
\item there exists a \emph{spotter time} $\widehat t<0$  such that $ \height_\infty(\gamma(\widehat t))>h_0$.
\end{enumerate}
\end{defi}

Critically, the third condition tells us that $\gamma$ intersects some $M\CW$ for $M\in \Mod$ at some time $t_M$ if and only if there is an associated spotter time  $\widehat{t}_M<t_M$ satisfying $\height_{\infty}M^{-1}\gamma(\widehat t_M)>h_0$, or equivalently $\height_{M\infty}\gamma(\widehat t_M)>h_0$.

\begin{defi}
\label{defi:markable}
A geodesic $\gamma$ is \emph{markable} if it intersects $\Mod$-translates of $\CW$ infinitely many times in both the past and the future. Unless stated otherwise, we will also assume that $\gamma(0)\in \CW$.
\end{defi}

In the following lemma, we will show that, for markable geodesics, spotter times follow a natural progression. That is, if we see a spotter time $\widehat t$ associated to an intersection time $t$, then we must move beyond $t$ before seeing the spotter time associated to any other intersection.

\begin{lemma}\label{lemma:spotterorder}
Let $\gamma$ be a markable geodesic, and $M, M'\in \Mod$. Suppose that $\gamma(a) \in M \CW$ and $\gamma(b)\in M'\CW$, attested by the corresponding spotter times $\widehat a, \widehat b$. Then these must alternate order: if $a<b$ then $\hat a< a< \hat b<b$.
\begin{proof}
We will prove an equivalent statement: if $\max(\widehat a, \widehat b)< \min(a, b)$, then $a=b$. Suppose it is false. Since $\gamma$ is markable, we may assume without loss of generality  that $\gamma(0)\in C_\W$, $0<\hat a<\hat b<\min(a,b)$. 

Let $\tee_{j}$ be the sequence in Lemma \ref{lemma:upshot}.  Then for some fixed $j$, we have $\tee_{j-1}<\hat a \leq \tee_j$. Conclusion 2 of the same lemma states that, since $\hat a$ is in the correct range and $\gamma(a)\in M\CW$, we have $M=M_{i_j}$ and by the definition of $\tee_j$  (that is, Conclusion 1 of the lemma) we have $a=\tee_j$. Furthermore, $\tee_{j-1}<\hat b<a=\tee_j$, so by the same argument $b=\tee_j$, as desired.
\end{proof}
\end{lemma}

We can now show that if a geodesic starts in $\CW$, its next intersection with a translate of $\CW$ will be captured by an iteration of the shift map.

\begin{lemma}
\label{lemma:isspeedup}
Let  $\gamma$ be  a markable geodesic such that $\gamma(0)\in \CW$, and suppose that the next intersection with a translate of $\CW$ occurs at $M\CW$. Then for some $j\ge 1$, we have $M=M_{\ai_j}$ and $\gamma(\tee_j)\in M\CW$, where $\ai_j,\tee_j$ are defined for $\gamma$ in Lemma \ref{lemma:upshot}.
\begin{proof}
Let $t>0$ denote the time when $\gamma(t)\in M\CW$. We know that there must exist a spotter time $\widehat t$ associated to $t$ and moreover, by Lemma \ref{lemma:spotterorder}, we know that $0<\widehat t< t$. Let $j\ge 1$ be such that $\tee_{j-1}\le \widehat t\le \tee_j$. Then by conclusion (2) of Lemma \ref{lemma:upshot}, we have that $M=M_{\ai_j}$ and $\gamma(\tee_j)\in M\CW$.
\end{proof}
\end{lemma}

We can now prove the Markable Geodesic Theorem.

\begin{proof}[Proof of Theorem \ref{thm:markable}]
For positive $i$, let $a_i$ and $M_i$ be the digits and mappings corresponding to the CF expansion of the forward endpoint $\gamma_+$, making property (2) immediate. We will define the remaining data iteratively. 

Let $t_1>0$ be the first positive time when $\gamma$ intersects an $\Mod$-translate of $\CW$. Lemma \ref{lemma:isspeedup} then provides an index $k$ such that $t_1=\tee_k$ and a corresponding number $\ai_k$ which we record as $i_1$ satisfying $\gamma(t_1)\in M_{i_1}\CW$. We will now show that properties (1), (4), and (3) hold on the initial segment $[t_0,t_1]$. 

Let $\hat{t}_1$ be a spotter time associated to the intersection of $\gamma$ with $M_{i_1}\CW$; that is, $\hat{t}_1<t_1$ and  $\height_{M_{i_1}\infty} \gamma(\hat{t}_1)>h_0>1$. 
Since $\gamma(t_0)\in\CW$ and $\gamma(t_1)\in M_{i_1}\CW$, then by Lemma \ref{lemma:spotterorder} we have that $\hat{t}_1\in [t_0,t_1]$. Let $\epsilon$ be the distance (not depending on $\gamma$) between the horospheres $\height_\infty(\cdot)=1$ and $\height_\infty(\cdot)=h_0$. Since $\gamma$ is a unit speed geodesic, $t_1-t_0>\epsilon$, and property (1) holds for $j=1$.

Next, the ``if'' direction of property (4) is immediate for $j=0$ and $j=1$ from the definitions. Now suppose $t\in(t_0,t_1]$ satisfies $\gamma(t)\in M\CW$ for some $M\in\Mod$. Then by definition of $t_1$ we have that $t=t_1$, and from Lemma \ref{lemma:upshot} we have that $M=M_{i_1}$. Thus the ``only if'' direction of property (4) holds for $t\in(t_0,t_1]$.

Suppose next that $t\in [t_0, t_1]$ satisfies $\height_{M\infty} \gamma(t)>h_0$ for some $M\in \Mod$. Then by Lemma \ref{lemma:upshot} there exists $\ell\ge 1$ and $t'>t$ such that $M=M_\ell$, and $\gamma(t')\in M_\ell \mathbb{W}$. By definition of $\CW$ via spotter times, we obtain that $\gamma(t')\in M_\ell\CW$. Since we assumed that $t_1$ is the first time that the forward ray of $\gamma$ intersects $\CW$, we have that $t_1\leq t'$. The converse inequality is given by Lemma \ref{lemma:spotterorder}, since $t$ is a spotter time associated to $t'$, so that $t_1=t'$ and $M=M_{i_1}$ follows from property (4). So property (3) holds for $j=1$.

To define $t_j,i_j$ for $j\geq 2$, we now consider a renormalized geodesic $\gamma'=M_{i_1}^{-1} \gamma$ with $\gamma'(0)=M_{i_1}^{-1} \gamma (t_1)$. We may then find $t'_1,i'_1$ for $\gamma'$ as we did above and let $t_2=t_1+t'_1$ and $i_2=i_1+i'_1$. Iterating this procedure gives $t_j,i_j$ for all $j\ge 1$. By the work above, properties (1), (3), and (4) hold on the corresponding initial segment of the renormalized geodesics and thus hold on the entire forward geodesic ray of $\gamma$.  Moreover from this definition, we see that property (5) holds for all $i,j,k$ that are non-negative.

To define $a_i, M_i$ for non-negative $i$ and $i_j,t_j$ for negative $j$, let $t_{-1}$ be the smallest (in norm) negative value for which $\gamma(t_{-1})$ intersects a $\Mod$-translate $M\CW$ of $\CW$. Consider a renormalized geodesic $\gamma'=M^{-1} \gamma$ with $\gamma'(0)= M^{-1}\gamma(t_{-1})\in \CW$. Set $i_{-1}=-i'_1$, $a_i=a'_{i+i_{-1}}$, and $M_i=M^{-1}M'_{i-i_{-1}}$ for $i_{-1}< i\le 0$. Since $\gamma'$ is a markable geodesic satisfying the conditions of the theorem and properties (1)--(4) hold for $\gamma'\vert_{[0,\infty]}$, so properties (1)--(5) hold for $\gamma\vert_{[t_{-1},\infty)}$. Iterating this process yields the remaining definitions and properties on the backwards ray of $\gamma$ (note that the full ray is covered by property (1)).
\end{proof} 

\section{Ergodicity}\label{sec:ergodicity}
We now prove the ergodicity of the shift map by first relating the cross-section $\CW$ studied in \S \ref{sec:markable} to geodesic flow on a quotient of $\Hyp$, and then to the shift map on the boundary. We start by recalling the ergodicity result for geodesic flow.
This section culminates in the ergodicity part of Theorem \ref{thm:main}.

\begin{remark}
All statements concerning ergodicity and measure will be made with respect to the relevant Hausdorff measure; depending on context this can be interpreted as Haar measure, surface measure, or Lebesgue measure.  Because there are no surprises along the way, we will suppress discussion of the details. 
\end{remark}

\subsection{Ergodicity of the Geodesic Flow}
The space $(\Hyp, d_\Hyp)$ is a symmetric space with a complete Riemannian metric with pinched negative curvature. In particular, any pair of points in $\Hyp$ (indeed, in $\overline{\Hyp}\cup\{\infty\}$) determines a unique geodesic. Alternately, a \emph{pointed} geodesic is determined by an element of the unit tangent bundle $T^1\Hyp$, namely a point in $\Hyp$ and a unit vector over it.

The geodesic flow on $T^1\Hyp$ moves vectors along geodesics as follows:
\begin{defi}[Geodesic Flow]
Given a vector $(h,v)\in T^1\Hyp$, let $\gamma: \R\rightarrow \Hyp$ be a unit-speed geodesic satisfying $\gamma(0)=h$ and $\gamma'(0)=v$.  The time-$t$ geodesic flow of $(h,v)$ is then given by $\phi_t(v):=(\gamma(t),\gamma'(t))\in T^1\Hyp$.
\end{defi}

Given a set $A\subset T^1\Hyp$, one says that $A$ is $\phi$-invariant, if for each $t\in \R$, the symmetric difference $(\phi_t^{-1}A) \triangle A$ has measure zero. We will be interested in sets $A$ that are furthermore invariant under a lattice $\Gamma\subset G$, i.e., $\mu(\gamma(A)\triangle A)=0$ for every $\gamma\in \Gamma$.

We can now state Mautner's Ergodicity Theorem (cf.~Moore's extension of the result to the frame bundle \cite{MR776417}):
\begin{thm}[Mautner's Ergodicity Theorem \cite{MR0084823}]
\label{thm:Mautner}
Let $\Gamma$ be a lattice in $G$, and $A\subset T^1\Hyp$ a $\Gamma$-invariant set that is furthermore invariant under geodesic flow. Then either $\mu(A)=0$ or $\mu(T^1\Hyp\setminus A)=0$.
\end{thm}

\subsection{Ergodicity of the Markable Cross-Section}\label{subsec:geodesicflow}

We continue working with a fixed complete, discrete, and proper Iwasawa continued fraction algorithm. Consider the natural projection $\pi_\Hyp: \Hyp \rightarrow \Mod\backslash \Hyp$.

Mautner's Theorem \ref{thm:Mautner} immediately applies to our setting. We record this in the following lemma, which can be interpreted either in the formulation of Theorem \ref{thm:Mautner} or, equivalently, using orbifold geodesic flow.
\begin{lemma}
Geodesic flow on $\Mod\backslash \Hyp$ is ergodic.
\begin{proof}
$\Mod$ is assumed to be discrete; to show it is a lattice we must show that there exists a finite-volume fundamental domain for $\Mod$. Let $K'$ be the region lying over both $K$ having horoheight at least $\epsilon>0$, for a choice of $\epsilon$ satisfying $\rad(K\times[0,\epsilon])^{-2}>1$. Given a point $h\in \Hyp$, we may use $\Zee$ to translate $h$ so that it lies over $K$, and invert it if necessary to increase its horoheight multiplicatively by at least $\rad(K\times[0,\epsilon])^{-2}$ (see \cite{1510.06033} for the interaction of horoheight and inversions), and translate again to place it over $K$. Within finitely many iterations, we obtain an image of $h$ contained in $K'$. Thus, $K'$ contains a fundamental domain for the $\Mod$ action on $\Hyp$. Lastly, $K'$ has horoheight bounded below and bounded extent along $\X$, so has finite hyperbolic volume. 
\end{proof}
\end{lemma}

\begin{lemma}\label{lemma:piergodic}
The first-return map on $\pi_\Hyp(\CW)$ is a.e.~well-defined and ergodic.
\begin{proof}
Consider the family $\mathcal F \subset T^1\Hyp$ of geodesic rays that pass through $\CW$. Recalling that $\CW$ consists of geodesics coming from large horoheight through the wall $\W$ and proceeding to $K$, it is clear $\mathcal F$ has positive measure. Since $\Mod$ is discrete, $\pi_\Hyp(\mathcal F)$ also has positive measure. Thus, by ergodicity, almost every geodesic in $\Mod\backslash \Hyp$ passes through $\pi_\Hyp(\CW)$. 

Since $\pi_\Hyp(\CW)$ is generically transverse to geodesic flow, we conclude that almost every geodesic ray in $\pi_\Hyp(\CW)$ returns to $\pi_\Hyp(\CW)$, and that the resulting first-return map is ergodic.
\end{proof}
\end{lemma}

We are now able to show that markable geodesics are generic:
\begin{cor}
\label{cor:markablegeneric}
Almost every geodesic $\gamma$ satisfying $\gamma(0)\in \CW$ is markable.
\begin{proof}
By the previous lemma, the first-return mapping on $\pi_\Hyp(\CW)$ is well-defined. Thus, given a generic geodesic ray $\gamma$ in $\CW$, $\pi_\Hyp(\gamma)$ will return to $\pi_\Hyp(\CW)$ after some time. Lifting to $\Hyp$, this implies that $\gamma$ intersects $M\CW$ for some $M\in \Mod$. Iterating the first-return map gives infinitely many intersections. Reversing the flow gives the same result for the backward orbit of $\gamma$.
\end{proof}
\end{cor}

Now that we have shown that almost all geodesics are markable, we can quickly prove that $\CW$ has no unexpected symmetries:

\begin{cor}\label{cor:piinjective}
The restriction of $\pi_\Hyp$ to $\CW$ is a.e.~injective.
\begin{proof}
Suppose the statement is false, and there exists a non-identity mapping $M\in \Mod$ such that $M\CW \cap \CW$ has positive measure. Then by the previous corollary there is a markable geodesic $\gamma$ with  $\gamma(0)\in M\CW \cap \CW$. But then we have $\gamma(0)\in \CW$ and $M\gamma(0)\in \CW$, and it follows from the Intersection Detection Property of Theorem \ref{thm:markable} that $M=M_{i_0}=\id$.
\end{proof}
\end{cor}

\begin{defi}
Let us define a mapping $\psi: \CW\rightarrow \CW$ by $\psi(\gamma)(t)=M_{i_1}^{-1}\gamma(t+t_1)$, where $M_{i_1}$ and $t_1$ are given by 
Theorem \ref{thm:markable}. This is well-defined almost everywhere.
\end{defi}

\begin{prop}\label{prop:phiergodic}
The mapping $\psi: \CW\rightarrow\CW$ is ergodic.
\begin{proof}
The first-return map on $\pi_\Hyp(\CW)$ is ergodic by Lemma \ref{lemma:piergodic}. Corollary \ref{cor:piinjective} then allows us to identify $\pi_\Hyp(\CW)$ with $\CW$, and Theorem \ref{thm:markable} tells us that $\psi$ is indeed a lift of the first-return mapping on $\pi_\Hyp(\CW)$.
\end{proof}
\end{prop}

\subsection{Ergodicity of a Natural Extension and of the Shift Map}\label{subsec:extension}

At this point, we would like to project $\CW$ onto the forward endpoint and use the ergodicity of $\psi$ to derive the ergodicity of $T$. However, the transformation that $\psi$ induces on the forward endpoint is a jump transformation associated to $T$ and it is not the case that the ergodicity of a jump transformation implies the ergodicity of the original transformation. (See, for example, Chapters 17--19 of \cite{SchweigerBook}.) So we will instead project onto both endpoints and analyze the resulting transformation more carefully.

Throughout the rest of this section, we will assume, without directly stating it, that all statements about sets hold up to sets of zero measure and that any geodesic under consideration is markable, since this is a generic condition. We continue to work with a complete, discrete, and proper Iwasawa CF expansion.

Let $\pi:\CW\to K\times \X$ be the injective map from a geodesic $\gamma$ intersecting $\CW$ to its forward and backward endpoints $(\gamma_+,\gamma_-)$. On $\pi(\CW)$, $\psi$ induces the isomorphic mapping $\Psi=\pi\circ \psi\circ \pi^{-1}$, which is ergodic on $\pi(\CW)$. Since, by the Markable Geodesic Theorem \ref{thm:markable}, $\psi$ acts on a geodesic $\gamma$ by the mapping $M_{i_1}$ associated to $\gamma_+$, we conclude that $\Psi(\gamma_+,\gamma_-)=(M_{i_1}^{-1}\gamma_+,M_{i_1}^{-1}\gamma_-)$. 

    Let us extend the shift map $T$ to act on $K\times \X$ by $\hat{T}(z,w)=(M_1^{-1} z,M_1^{-1} w)$ where $M_1\in\Mod$ is the mapping associated to $z$. Since $Tz=M_1^{-1}z$, this truly is an extension. Let $\overline{K}=\cup_{i=0}^\infty \hat{T}^i \pi(\CW)\subset K\times \X$. 

We wish to compare how $\Psi$ acts on $\pi(\CW)$ with how $\hat{T}$ acts on $\overline K$.  In the following lemma, will show that the restriction $\htk$ of $\hat{T}$ to $\overline{K}$ is well-behaved.

\begin{lemma}
\label{lemma:Tinvariance}
$\htk: \overline K \rightarrow \overline K$ is surjective. Furthermore, a.e.~point of $\overline K$ returns to $\pi(\CW)$ within finitely many iterations of $\htk$, so that we have \[\label{eq:overlineKdef} \overline{K} = \bigcup_{i=0}^\infty \htk^{-i} \pi(\CW).\]
\begin{proof}
It is immediate from the definition of $\overline K$ that $\htk\overline K \subset \overline K$. To prove the reverse containment, we  wish to show that for any $(z,w)\in\overline K$, there exists $(z',w')\in \overline K$ with $\htk(z',w')=(z,w)$.

Since $(z,w)\in \overline{K}$, there exists a smallest non-negative integer $i$ such that $(z,w)\in \htk^i \pi(\CW)$. If $i\ge 1$, then clearly there is $(z',w')\in \htk^{i-1}\pi(\CW)$ such that $\htk(z',w')=(z,w)$. 

So suppose $i=0$. Then $(z,w)\in \pi(\CW)$. Since $\Psi$ is an onto map of $\pi(\CW)$ to itself, for a.e.~$(z,w)$ there exists some $(z'',w'')$ such that $\Psi(z'',w'')=(z,w)$. Thus, if we let $i_1$ be the index so that $\Psi(z'',w'')=(M_{i_1}^{-1}z'',M_{i_1}^{-1} w'')=\htk^{i_1}(z'',w'')$, then we have that $(z,w)\in \htk^{i_1} \pi(\CW)$ with $i_1>0$ and the argument of the previous paragraph applies.

Implicit in the last paragraph is the idea that for a.e.~$(z,w)\in \pi(\CW)$, $\Psi(z,w)\in \pi(\CW)$ as well, so that $(z,w)$ returns to $\pi(\CW)$ in a finite number of iterations of $\htk$. Since every $(z,w)\in \overline{K}\setminus \pi(\CW)$ appears in some $\htk^i \pi(\CW)$, say, $\htk^i (z'',w'')=(z,w)$, we can also extend this to say that a.e.~point in $\overline{K}$ returns to $\pi(\CW)$ under a finite number of iterations.

This immediately shows that $\overline{K} \subset \bigcup_{i=0}^\infty \htk^{-i} \pi(\CW)$ and the reverse inclusion is trivial.
\end{proof}
\end{lemma}

We restrict our attention to $\overline{K}$, setting $\hat{T}:=\htk$.

The equation \eqref{eq:overlineKdef} looks similar to the definition of a natural extension, so raises the following question, which we will not address:
\begin{question}
Is $\hat{T}: \overline K\rightarrow \overline K$ the natural extension of $T: K\rightarrow K$?
\end{question}
One can look at, for example, \cite{EINN} for a discussion of the natural extension in the case of the A. Hurwitz complex CF.

Now we can state the connection between $\Psi$ and $\hat{T}$:

\begin{lemma}\label{lemma:induced}
$\Psi$ is the transformation induced by restricting $\hat{T}$ to $\pi(\CW)$.
\begin{proof}
Since $\Zee$ is countable, the set of points in $K$ with eventually periodic continued fraction expansions is countable as well, and hence, since we are working up to measure zero, we may assume any points under consideration are not eventually periodic.

Let $(z,w)\in\pi(\CW)$ and let $i(z,w)$ be the minimal positive integer such that $\hat{T}^{i(z,w)}(z,w)\in\pi(\CW)$. The existence of $i(z,w)$ a.e.~follows from Lemma \ref{lemma:Tinvariance}. We wish to show that, where it exists, $\hat{T}^{i(z,w)}(z,w)=\Psi(z,w)$. 

Let $\gamma$ be the markable geodesic with endpoints $(z,w)$, and let $i_1$ be the corresponding value from the marking in Theorem \ref{thm:markable}. Then $\Psi(z,w)=(M_{i_1}^{-1}z,M_{i_1}^{-1}w)$ and thus $\hat{T}^{i_1}(z,w)=\Psi(z,w)\in\pi(\CW)$. By the minimality of $i(z,w)$, we have that $i(z,w)\le i_1$. We must show that $i(z,w)$ cannot be strictly less than $i_1$.

Suppose $i(z,w)<i_1$ and consider the mapping  $M=M_{i(z,w)}$. Since $(M^{-1}z,M^{-1}w)\in\pi(\CW)$, $M^{-1}\gamma$ intersects $\CW$. This means $\gamma$ intersects $M\CW$ and thus by the Intersection Detection property of Theorem \ref{thm:markable}, $M=M_{i_j}$ for some $j$. Since the two mappings are equal, we have that $T^{i(z,w)}z=M_{i(z,w)}^{-1}z=M_{i_j}^{-1} z=T^{i_j} z$. But since we have assumed $z$ does not have an eventually periodic expansion, this is only possible if $i(z,w)=i_j$. And since there are no positive $i_j$ between $0$ and $i_1$, we must have that $i(z,w)=i_1$, which completes the proof.
\end{proof}
\end{lemma}

 While there is a close connection between the dynamical properties of a map and the dynamical properties a new map induced from the first, in general one cannot use the ergodicity of the induced map to conclude the ergodicity of the original map; however, Lemma \ref{lemma:induced} when combined with \eqref{eq:overlineKdef} is enough to prove the following result immediately (see Theorem 17.2.4 of \cite{SchweigerBook} for full details).

 \begin{lemma}\label{lemma:kbarergodic}
$\hat{T}$ is ergodic on $\overline{K}$.
\end{lemma}

We can now project to the first coordinate to complete the proof of Theorem \ref{thm:main} (see also \S \ref{subsec:condensed proof}):
\begin{proof}[Proof of Theorem \ref{thm:main}]
Let us suppose the shift map is not ergodic. Then there are complementary subsets $A$ and $B$ of $K$ that are both invariant under $T$ and have non-zero measure. We may extend these to complementary subsets $A', B'$ of $\overline K$ by taking their preimages under projection to the first coordinate. Both $A'$ and $B'$ have positive measure since $\pi(\CW)\subset \overline K$ and we claim there exists a neighborhood $U$ of infinity in $\X$ such that $K\times U\subset \pi(\CW)$.

Let us now show that this set $U$ does exist. Consider any pair $(\gamma_+,\gamma_-)$ of endpoints of a geodesic $\gamma$, such that $\gamma_+\in K$ and $\norm{\gamma_-}$ is sufficiently large. In particular, if $\norm{\gamma_-}>1+\epsilon$ with $\epsilon$ as in Lemma \ref{lemma:horocap}, then the conclusion of that lemma and the definition of $\W$ imply that the geodesic $\gamma$ passes through $\W$. Moreover, by taking the framework of Lemma \ref{lemma:horocap} and dilating, we see that if $\norm{\gamma_-}$ is sufficiently large, then the geodesic must travel far into the cusp at infinity: namely, there must exist a time $\hat{t}$ such that $\height_\infty(\gamma(\hat{t}))>h_0$. Thus, $\gamma$ does intersect $\CW$ and $(\gamma_+,\gamma_-)\in\pi(\CW)$ as desired. (Since we are working up to measure zero sets, we may assume that $\gamma_+\not\in \Mod(\{\infty\}\cup E)$ as well.)

Consider $\hat{T}^{-1}A'$. Any point $(z,w)\in \overline{K}$ such that $\hat{T}(z,w)\in A'$ must clearly satisfy $Tz\in A$. In other words $z\in T^{-1} A=A$. Thus $(z,w)\in A'$, so $\hat{T}^{-1}A'\subset A'$ and likewise $\hat{T}^{-1}B'\subset B'$. Hence $A'$ and $B'$ are both disjoint $T$-invariant subsets of $\overline K$ with positive measure. The ergodicity of $\hat T: \overline K\rightarrow \overline K$ provided by Lemma \ref{lemma:kbarergodic} gives the contradiction.
\end{proof}

\begin{remark}
We have proved ergodicity with respect to Lebesgue measure, but with the framework we have developed, we may now consider the question of absolutely continuous invariant measures as well.

First, note that since geodesic flow preserves Haar measure on $\Hyp$, there is a canonical derivation of an invariant measure for $\psi$ on $\CW$. This then projects to an invariant measure for $\Psi$ on $\pi(\CW)$. Since $\Psi$ is the transformation induced by restriction $\hat{T}$ to $\pi(\CW)$, there is again a canonical derivation of an invariant measure for $\hat{T}$ on $\overline{K}$ (see \cite[Thm.~17.1.6]{SchweigerBook}). From here projection onto the first coordinate would give an invariant measure for $T$ on $K$. All of these operations preserve the fact that they are absolutely continuous with respect to the corresponding Hausdorff measure.

 Note that even though the measure on $\CW$ and $\pi(\CW)$ is bounded, the measure on $\overline K$ and $K$ may be infinite. Indeed, this occurs for the Rosen continued fractions \cite{GH}.
\end{remark}

\subsection{Application: Ergodic components of Incomplete Iwasawa CFs}
In this subsection we will prove Theorem \ref{thm:notmain}.

Let $\mathcal{R}$ denote the set of central symmetries of $\Mod$ (cf.~Definition \ref{defi:centrallySymmetric}).

\begin{lemma}\label{lemma:quasicommutative}
Let $r\in \mathcal{R}$. Then for any $a\in \Zee$ there exists $a'\in \Zee$, $r'\in \mathcal{R}$ such that $a\iota r = r'a'\iota$. Moreover if $r'$ is the identity, then $r$ must be as well.
\end{lemma}

\begin{proof}
Since $a\iota r \iota^{-1}\in\Stab_\Mod(\infty)$, the decomposability assumption on $\mathcal{R}$ implies that there exist $r'\in\mathcal{R}$ and $a'\in \Zee$ such that $a\iota r \iota^{-1}\iota=r'a'\iota$, as desired.

Let $r''$ denote $\iota r \iota^{-1}$. Since this fixes $0$ and $\infty$, it must belong to $\mathcal{R}$. So if $r'$ is the identity, then $r'a'=ar''$ implies that $a^{-1}a'=r''$. But $\mathcal{R}\cap \Zee=\{\id\}$, so $r''$ and hence $r$ must be the identity.
\end{proof}

At this point we wish to start connecting the behavior of an incomplete Iwasawa CF with $n$ central symmetries with the behavior of its completion.

As such let us specialize our notation. Let $K$ be the symmetric fundamental domain for the incomplete continued fraction over $\Zee$ and let $K_c$ be an associated fundamental domain for the \emph{completion} of the continued fraction over $\Stab_\Mod(\infty)$ so that $K=\bigcup_{r\in\mathcal{R}}r K_c$ up to a set of measure zero. Let $T$ be the shift map on $K$ that acts by $\iota$ and then an element of $\Zee$. Let $T_c$ be the shift map on $K_c$ that acts by $\iota$ and then an element of $\Stab_\Mod(\infty)$.

\begin{lemma}
With the notation of the paragraph directly above, the  map $T$ on $K$ is isomorphic to a skew-product $T_c\rtimes f$ on $K_c\times \mathcal{R}$ over the  map $T_c$ on $K_c$.
\begin{proof}
There is an obvious isomorphism between  $K_c\times \mathcal{R}$ and $K$ given by $(z,r)\leftrightarrow rz$. The map $T$ acts on $rz$ by $a\iota$ for some $a\in \Zee$. By Lemma \ref{lemma:quasicommutative}, there exists $a'\in \Zee, r'\in R$ such that $T(rz)=r'a'\iota (z)$. Let $r''$ be such that $r''a'\iota(z)\in K_c$, so that $T$ can be considered as acting on the space $K_c\times R$ by 
\(
(z,r)\mapsto (r''a'\iota z, r'r''^{-1}).
\)
Since $r''a'\in \Stab_\Mod(\infty)$, this maps $(z,r)$ to $T_c(z)$ in the first coordinate. Let $f(z,r)=r'r''^{-1}$, so that $T=T_c\rtimes f$. To show that $T_c\rtimes f$ is truly a skew-product and finish the proof, we must show that for almost all fixed $z$, $f(z,\cdot)$ is an injection (and hence a bijection).

Suppose that $f(z,\cdot)$ is not an injection, so that $r_1\neq r_2$ but $f(z,r_1)=f(z,r_2)$. This implies that $T(r_1z)=T(r_2 z)$. Let  $a_1,a_2\in \Zee$ be such that $T$ acts by $a_1\iota $ on $r_1z$ and acts by $a_2 \iota$ on $r_2 z$. Then $a_1\iota r_1 \iota^{-1} (\iota z)=a_2\iota r_2 \iota^{-1} (\iota z)$. But for almost all $z$ (namely, those $z$ not belonging to the exceptional set $E$ \eqref{eq:Edefn}), $a_1\iota r_1\iota^{-1}$ is the unique element of $\Stab_\Mod(\infty)$ that brings $\iota z$ to $K$. Thus, for such $z$, $a_1\iota r_1 \iota^{-1} =a_2\iota r_2 \iota^{-1}$. Recall from the proof of the previous lemma that $\iota r_1\iota^{-1},\iota r_2\iota^{-1}\in \mathcal{R}$. So by the uniqueness of the decomposition, we have that $\iota r_1 \iota^{-1}=\iota r_2\iota^{-1}$, and hence $r_1=r_2$. So $f(z,\cdot)$ is injective.
\end{proof}
\end{lemma}

Theorem \ref{thm:notmain} immediately follows from the next lemma:

\begin{lemma}
Let $A$ be any ergodic component of $K$ with positive measure, then the measure of $A$ must be at least $1/|\mathcal{R}|$ (all with respect to a normalized Lebesgue measure on $K$).
\begin{proof}
We may consider $A$ as a positive measure subset of $K_c\times \mathcal{R}$ invariant under the skew-product  $T_c \rtimes f$ defined in the previous lemma. Consider also the standard projection onto the first coordinate: $\pi_{K}:K_c\times R\to K_c$. Since $T_c$ is the shift map associated to a discrete, proper, and \emph{complete} Iwasawa CF expansion, it will be ergodic due to Theorem \ref{thm:main}, and thus it suffices to prove that $\pi_K(A)$ is a $T_c$-invariant set, since it must have full measure on $K_c$ (i.e., $1/|\mathcal{R}|$).

Suppose $z\in \pi_K(A)$, so that there exists $r\in \mathcal R$ such that $(z,r)\in A$. Let $z'\in T_c^{-1}z$. Then, since $T_c\rtimes f$ is a skew-product, there exists (for almost all such $z$) $r' \in \mathcal R$ such that $(T_c\rtimes f)(z',r')=(z,r)$. Thus $(z',r')\in (T_c\rtimes f)^{-1} A = A$, so $z'\in \pi_K(A)$. Thus $T_c^{-1}\pi_K(A)$ is (up to measure zero), a subset of $\pi_K(A)$.

Now suppose $z\in \pi_K(A)$ and again let $r\in \mathcal R$ be such that $(z,r)\in A=T^{-1} A$. Thus $T(z,r)\in A$, and projecting this into the first coordinate, we see that $T_c z\in \pi_K(A)$. Thus $\pi_K(A)\subset T_c^{-1}\pi_K(A)$. This proves the two sets are equal up to measure zero, as desired.
\end{proof}
\end{lemma}

In certain cases one can show that the skew-product over an ergodic transformation is itself ergodic, see  \cite{Vmatrix} and related papers of the second author for some interesting examples. If we could prove such a result here, we could remove the completeness condition in the case of centrally symmetric systems.

\subsection{Application: Tail Equivalence}\label{sec:tail}

In this section we prove Theorem \ref{thm:IntroTail} in the following more precise formulation (note that markable geodesics are generic by Corollary \ref{cor:markablegeneric}):

\begin{thm}[Tail equivalence of markable geodesics]
\label{thm:tail}
Let $\gamma$ be a markable geodesic and $\gamma'=M\gamma$ with $M\in\Mod$ and $\gamma'_+\in K$. If $a_i, a'_i$ are the sequence of CF digits of $\gamma_+$ and $\gamma'_+$, respectively, then they have the same tail---i.e., there exist some $k,k'\in\mathbb{N}$ such that $a_{k+i}=a'_{k'+i}$ for all $i\ge 1$. 
\begin{remark}
We note that the condition $\gamma'_+\in K$ is not necessary. If it were not there, we could define $a'_0=[{\gamma'_+}]$ and let the continued fraction expansion of $\gamma'_+$ start with this $a'_0$; however, since this $a'_0$ might be confused with the corresponding digit of the marking, we will not use it here.
\end{remark}
\begin{proof}
While $\gamma'$ is a markable geodesic, it may or may not pass through $\CW$.

The result follows immediately from Theorem \ref{thm:markable} if $\gamma'$ \emph{does} pass through $\CW$: the Cusp Detection Property gives us that for some $j$, $M=M_{i_j}^{-1}$. So the marking of $\gamma'$ is a shift of the marking of $\gamma$. If $j\ge 0$, then $a'_i=a_{i_j+i}$ for $i\ge 1$, and if $j<0$, then $a'_{-i_j+i}=a_i$ for $i\ge 1$.

We now assume that $\gamma'$ does not pass through $\CW$. If $\norm{\gamma'_-}\ge 1+\epsilon$, with $\epsilon$ as in Lemma \ref{lemma:horocap}, then we apply Lemma \ref{lemma:finiteintersection} to see that $\gamma'$ intersects $\W$. Let $\gamma''(t)=\gamma'(t+t')$ be such that $\gamma''(0)\in\W$. On the other hand, if $\norm{\gamma'_-}<1+\epsilon$, then we may apply the proof of Lemma \ref{lemma:exile} to $\gamma'$ to find an index $\ai_1$ and corresponding time $\tee_1$ such that $M_{\ai_1}^{-1}\gamma'(\tee_1)\in\W$. (Note that the condition in the lemma that $\gamma(0)\in\W$ is not actually used in the proof, only that $|\gamma(0)|<1+\epsilon$. Moreover, since $\gamma'$ is markable, we know that $\gamma'_+\not\in \Mod\infty$.) In this case, let $\gamma''(t)=M_{\ai_1}^{-1}\gamma'(t+\tee_1)$, so that once again $\gamma''(0)\in \W$.

We claim that $\gamma'_+$ and $\gamma''_+$ are tail-equivalent. This is obvious in the first case, since $\gamma'_+=\gamma''_+$. In the second case, they are still tail-equivalent, since $\gamma''_+=T^{\ai_1}\gamma'_+$ and $T$ again acts via a shift of the digits. Moreover, $\gamma''$ is still a markable geodesic, since this property is $\Mod$-invariant. 

By applying the idea of the proof of Lemma \ref{lemma:isspeedup}, we have that $\gamma''$ intersects $M_{\ai_j}\CW$ at time $\tee_j$ for some $j$. In particular, if we let $\gamma'''(t)=M_{\ai_j}^{-1}\gamma''(t+\tee_j)$, then by the same argument as previously, we see that $\gamma'''_+$ is tail-equivalent to $\gamma''_+$ and hence to $\gamma'_+$. In addition, $\gamma'''$ now passes through $\CW$ so our earlier argument applies and we see that $\gamma'''_+$ is tail-equivalent to $\gamma_+$, as desired.
\end{proof}
\end{thm}

\bibliographystyle{amsplain}
\bibliography{bib}

\end{document}